\documentclass[12pt]{amsart}
\usepackage{amscd, amsfonts, amssymb, amsthm}
\usepackage{parskip}
\usepackage[all]{xy}
\usepackage[margin=1in,includeheadfoot]{geometry}

\newcommand\CA{{\mathcal A}} 

\newcommand\CC{{\mathcal C}}

\newcommand\BBC{{\mathbb C}}

\newcommand\BBR{{\mathbb R}}

\newcommand\codim{\operatorname{codim}}

\newcommand\Fix{{\operatorname{Fix}}}

\newcommand\Ind{{\operatorname{Ind}}}
\newcommand\GL{\operatorname{GL}}

\newcommand\rad{{\operatorname{rad}}}

\newcommand\rk{{\operatorname{rk}}}

\newcommand\U{\operatorname{U}}

\newcommand\inverse{^{-1}}
\renewcommand\th{{^{\text{th}}}}

\newcommand\sh{{\operatorname{sh}}}

\numberwithin{equation}{section}

\theoremstyle{plain}
\newtheorem{lemma}[equation]{Lemma}
\newtheorem{theorem}[equation]{Theorem}
\newtheorem{conjecture}[equation]{Conjecture}
\newtheorem{corollary}[equation]{Corollary}
\newtheorem{proposition}[equation]{Proposition}

\thanks{The authors would like to thank their charming wives for their
  unwavering support during the preparation of this paper.}


\begin{document}

\title [Coxeter arrangements and Solomon's descent algebra] {Cohomology of
  Coxeter arrangements and Solomon's descent algebra}


\author[J.M. Douglass]{J. Matthew Douglass} \address{Department of
  Mathematics, University of North Texas, Denton TX, USA 76203}
\email{douglass@unt.edu} \urladdr{http://hilbert.math.unt.edu}

\author[G. Pfeiffer]{G\"otz Pfeiffer} \address{School of Mathematics,
  Statistics and Applied Mathematics, National University of Ireland
  Galway, Galway, Ireland}
\email{goetz.pfeiffer@nuigalway.ie}
\urladdr{http://schmidt.nuigalway.ie/~goetz}

\author[G. R\"ohrle]{Gerhard R\"ohrle} \address {Fakult\"at f\"ur
  Mathematik, Ruhr-Universit\"at Bochum, D-44780 Bochum, Germany}
\email{gerhard.roehrle@rub.de}
\urladdr{http://www.ruhr-uni-bochum.de/ffm/Lehrstuehle/Lehrstuhl-VI}


\begin{abstract}
  We refine a conjecture by Lehrer and Solomon on the structure of the
  Orlik-Solomon algebra of a finite Coxeter group $W$ and relate it to the
  descent algebra of $W$.  As a result, we claim that both the group algebra
  of $W$, and the Orlik-Solomon algebra of $W$ can be decomposed into a sum
  of induced one-dimensional representations of element centralizers, one
  for each conjugacy class of elements of $W$.  We give a uniform proof of
  the claim for symmetric groups. In addition, we prove that a relative
  version of the conjecture holds for every pair $(W, W_L)$, where $W$ is
  arbitrary and $W_L$ is a parabolic subgroup of $W$, all of whose
  irreducible factors are of type $A$.
\end{abstract}

\maketitle
\setcounter{tocdepth}{1}
\tableofcontents


\section{Introduction}

Suppose $V$ is a finite-dimensional, complex vector space. A linear
transformation $t$ in $\GL(V)$ is called a \emph{reflection} if it has
finite order and the fixed point set of $t$ is a hyperplane in $V$, or
equivalently, if $t$ is diagonalizable with eigenvalues $1$, with
multiplicity $\dim V-1$, and $\zeta$, where $\zeta$ is a root of unity, with
multiplicity $1$. Suppose that $W$ is a finite subgroup of $\GL(V)$
generated by a set of reflections $T$. For each $t$ in $T$, let $H_t=
\Fix(t)$ denote the fixed point set of $t$ in $V$ and set $\CA=\{\, H_t\mid
t\in T\,\}$ and $M_W= V\setminus \cup_{t\in T} H_t$. Then $(V, \CA)$ is a
hyperplane arrangement, and the complement $M_W$ is an open, $W$-stable
subset of $V$.

The action of $W$ on $M_W$ determines a representation of $W$ on the
singular cohomology of $M_W$. For $p\geq0$ let $H^p(M_W)$ denote the $p\th$
singular cohomology space of $M_W$ with complex coefficients and let
$H^*(M_W)= \bigoplus_{p\geq 0} H^p(M_W)$ denote the total cohomology of
$M_W$, a graded $\BBC$-vector space. It follows from a result of Brieskorn
\cite{brieskorn:tresses} that $\dim H^*(M_W)= |W|$ and so a naive guess
would be that the representation of $W$ on $H^*(M_W)$ is the regular
representation of $W$. A simple computation for the symmetric group $S_3$
shows that this is not the case.

In 1987, Lehrer \cite{lehrer:poincare} determined the character of the
representation of $W$ on $H^*(M_W)$ when $W=S_n$ is a symmetric group by
explicitly computing the ``equivariant Poincar\'e polynomials''
$P_{S_n}(g,t)= \sum_{p\geq0} \operatorname{trace}(g, H^p(M_W))t^p$, for $g$
in $S_n$ (here, $t$ is an indeterminate). Subsequently, equivariant
Poincar\'e polynomials were computed case by case for other groups by
various authors. In 2001, Blair and Lehrer \cite{blairlehrer:cohomology}
showed that for any complex reflection group, the equivariant Poincar\'e
polynomials have the form $P_W(g,t)= \sum_{x\in Z_W(g)} f_g(x) (-t)^{c(x)}$
where $f_g\colon Z_G(g)\to \BBC$ is explicitly given and $c(x)$ is the
codimension of the fixed point space of $x$ in $V$.  Felder and Veselov
\cite{felderveselov:coxeter} have found an elegant description of the
character of $H^*(M_W)$ when $W$ is a Coxeter group that precisely describes
how the character of $H^*(M_W)$ differs from the regular character $\rho$ of
$W$. Specifically, Felder and Veselov show that the character of $H^*(M_W)$
is given as
\[
\sum_{t} (2\,\Ind_{\langle t \rangle}^W (1) -\rho),
\]
where the sum runs over a set of ``special'' involutions $t$ in $W$.

In contrast, while the representation of $W$ on $H^*(M_W)$ is
well-understood, much less is known about the representations of $W$ on the
individual graded pieces $H^p(M_W)$ for $p\geq0$. When $W$ is a symmetric
group, Lehrer and Solomon \cite{lehrersolomon:symmetric} have described these
representations as sums of representations induced from linear characters of
centralizers of elements in $W$. They conjecture that a similar
decomposition exists in general.

For the symmetric group $S_n$, Barcelo and Bergeron
\cite{barcelobergeron:orlik-solomon} construct an explicit $S_n$-stable
subspace of the exterior algebra of the free Lie algebra on $n$ letters that
affords the character of $H^*(M_W)$ tensored with the sign character.  Their
construction could be used to study the characters of the individual
cohomology spaces $H^p(M_W)$.

For the hyperoctahedral group $W(B_n)$, Douglass \cite{douglass:cohomology}
extended the construction of Lehrer and Solomon and expressed each
$H^p(M_W)$ as a sum of representations induced from linear characters of
subgroups. However, the subgroups appearing are not always centralizers of
elements of $W(B_n)$.  At the same time, Bergeron used the free Lie algebra
on $2n$ letters to construct a representation of $W(B_n)$ analogous to the
representation of $S_n$ constructed in
\cite{barcelobergeron:orlik-solomon}. The character of this representation
of $W(B_n)$ is again the character of $H^*(M_W)$ tensored with the sign
character. He then uses this construction to study the characters of the
individual cohomology spaces $H^p(M_W)$.

In this paper we state a conjecture for a finite Coxeter group $W$
(Conjecture~\ref{conj}) that both refines the conjecture of Lehrer and
Solomon \cite[Conjecture 1.6]{lehrersolomon:symmetric} and directly relates
the representation of $W$ on $H^p(M_W)$ to a subrepresentation of the right
regular representation of $W$. It is straightforward to see that
Conjecture~\ref{conj} holds for $W$ if and only if it holds for every
irreducible factor of $W$. Thus, to prove the conjecture we may assume that
$W$ is irreducible. Conjecture~\ref{conj} is proved for symmetric groups in
this paper (Theorem~\ref{thm:main}). The conjecture has been proved for all
rank two Coxeter groups \cite{douglasspfeifferroehrle:inductive} and has
been checked using the computer algebra system {\sf GAP}~\cite{gap3} and the
package {\sf CHEVIE}~\cite{chevie} for all Coxeter groups with rank six or
less (\cite{bishopdouglasspfeifferroehrle:computations},
\cite{bishopdouglasspfeifferroehrle:computationsII}). More generally, in
\S\ref{sec:rel} we extend the constructions used in the proof of
Theorem~\ref{thm:main} and prove a ``relative'' version of
Conjecture~\ref{conj} for pairs $(W, W_L)$, where now $W$ is any finite
Coxeter group and $W_L$ is a parabolic subgroup of $W$, all of whose
irreducible factors are of type $A$. If the conclusion of
Theorem~\ref{thm:rel} were to hold for every parabolic subgroup $W$, not
just those that are products of symmetric groups, then Conjecture~\ref{conj}
would hold for $W$.  The statement of the conjecture, along with an
expository review of the background material from the theories of Coxeter
groups and hyperplane arrangements we use later in the paper, is given in
\S\ref{sec:expl}.

The subrepresentations of the right regular representation of $W$ that we
consider arise from a decomposition of a subalgebra of the group algebra of
$W$, known as ``Solomon's descent algebra,'' into projective, indecomposable
modules. Projective, indecomposable modules in an artinian $\BBC$-algebra
are generated by idempotents. The idempotents in the descent algebra we use
in this paper were discovered by Bergeron, Bergeron, Howlett, and Taylor
\cite{bergeronbergeronhowletttaylor:decomposition}; we call them \emph{BBHT
  idempotents.} In \S\ref{prelim} we study the relationships between BBHT
idempotents for $W$ and BBHT idempotents for parabolic subgroups. We also
compare the BBHT idempotents for $W$ with those of its irreducible factors
when $W$ is reducible. In \S\ref{induced} we show that the right ideals in
the group algebra of $W$ generated by BBHT idempotents afford induced
representations. It then follows that the BBHT idempotents give rise to a
decomposition of the group algebra that is the direct analog of the
decomposition of $H^*(M_W)$ given by Brieskorn's Lemma \cite[Lemme
3]{brieskorn:tresses}.

The proof of Conjecture~\ref{conj} for symmetric groups is given in
\S\ref{lambda=n} and \S\ref{arbitrary}. As a consequence, we obtain a
decomposition of the group algebra $\BBC S_n$ as a direct sum of
representations induced from one-dimensional representations of
centralizers, one for each conjugacy class. Similar decompositions of the
group algebra of the symmetric group have been proved independently by
Bergeron, Bergeron, and Garsia \cite{bergeronbergerongarsia:idempotents},
Hanlon \cite{hanlon:action}, and more recently, Schocker
\cite{schocker:hoheren}, all using different methods.

For readers familiar with the literature on the free Lie algebra and its
connections with the descent algebras and group algebras of symmetric
groups, Theorems~\ref{top}(a) and~\ref{thm:main}(a) will seem familiar.
Indeed, Theorem~\ref{top}(a) was likely known in some form to
Brandt~\cite{brandt:free}, Wever~\cite{wever:ueber}, and
Klyachko~\cite{klyachko:lie}, and a proof of Theorem~\ref{thm:main}(a) can
be extracted from results in~\cite[\S4]{bergeronbergerongarsia:idempotents},
\cite[Theorem 8.24]{reutenauer:free},
and~\cite[\S4]{garsiareutenauer:decomposition}.  In contrast with these
references, where the methods are combinatorial and the emphasis is on the
connections between the group algebra $\BBC S_n$ and the free Lie algebra on
$n$ letters, our methods are mainly group- and representation-theoretic,
using the theory of Coxeter groups in conjunction with special features of
symmetric groups, and our focus is on the connections between the group
algebra $\BBC S_n$ and the representation of $S_n$ on the cohomology spaces
$H^p(M_W)$. Moreover, our line of reasoning, which is motivated by the
arguments in Lehrer and Solomon \cite{lehrersolomon:symmetric}, demonstrates
a striking parallelism between the Orlik-Solomon algebras and group algebras
of symmetric groups that to our knowledge has not been observed before. We
hope that the approach taken in this paper will lead to a deeper
understanding of the topological and geometric properties of general Coxeter
arrangements.  For example, as we show in \S\ref{sec:rel}, the constructions
in~\S\ref{lambda=n} and~\S\ref{arbitrary} have natural extensions when the
focus is shifted from considering not just a single symmetric group to
considering a product of symmetric groups embedded as a parabolic subgroup
in a larger finite Coxeter group $W$, where the type of $W$ is arbitrary.

Bergeron and Bergeron conjecture in \cite{bergeronbergeron:orthogonal} that
there might be a decomposition of the group algebra $\BBC W(B_n)$ analogous
to the decomposition of $\BBC S_n$ studied by Bergeron-Bergeron-Garsia
\cite{bergeronbergerongarsia:idempotents} and Hanlon
\cite{hanlon:action}. In \cite{bergeron:hochschild} Bergeron gives a
decomposition of the group algebra of a hyperoctahedral group as a direct
sum of representations induced from linear characters of subgroups. However,
this decomposition is not the decomposition proposed in
Conjecture~\ref{conj}; it is the analog for group algebras of
hyperoctahedral groups of the decomposition of $H^*(M_W)$ found in
\cite{douglass:cohomology}.

\section{The Orlik-Solomon algebra and Solomon's descent algebra}
\label{sec:expl} 

In the rest of this paper we assume that $V$ is a finite-dimensional,
complex vector space and $W\subseteq \GL(V)$ is a finite Coxeter group with
Coxeter generating set $S$. Then each $s$ in $S$ acts on $V$ as a reflection
with order two, and $W$ is generated by $S$ subject to the relations
$(st)^{m_{s,t}}=1$, where $m_{s,s}=1$ and $m_{s,t}= m_{t,s}>1$ for $s\ne t$
in $S$. Let $T$ denote the set of all reflections in $W$.

We assume also that a positive, definite, Hermitian form
$\langle\,\cdot\,,\, \cdot\,\rangle$ on $V$ is given and that $W$ is a
subgroup of the unitary group of $V$ with respect to this form. It is known
that $\Fix(W)^\perp$, the orthogonal complement of the space of fixed points
of $W$ on $V$, has a basis $\Delta=\{\, \alpha_s\mid s\in S\,\}$ so that
$\langle \alpha_s, \alpha_t\rangle= -\cos (\pi/m_{s,t})$ for $s$ and $t$ in
$S$.  Then $s$ acts on $V$ as the reflection through the hyperplane
orthogonal to $\alpha_s$, and $\Phi=\{\, w(\alpha_s)\mid w\in W, s\in S\,\}$
is a root system in $\Fix(W)^\perp$ with base $\Delta$.

\subsection{Shapes and conjugacy classes} \label{subsect:conj} 

We begin by recalling a parameterization of the conjugacy classes in $W$ due
to Geck and Pfeiffer (see \cite[\S3.2]{geckpfeiffer:characters}) in a form
compatible with the arrangement $(V, \CA)$ of $W$.

The \emph{lattice of $\CA$}, denoted by $L(\CA)$, is the set of subspaces of
$V$ that arise as intersections of hyperplanes in $\CA$:
\[
L(\CA)=\{\, H_{t_1} \cap H_{t_2} \cap \dotsm \cap H_{t_p}\mid t_1, t_2,
\dots, t_p\in T\,\}.
\]
For $X$ in $L(\CA)$ define
\[
W_X= \{\, w\in W\mid X\subseteq \Fix(w)\,\}
\]
to be the pointwise stabilizer of $X$ in $W$. It follows from Steinberg's
Theorem \cite{steinberg:differential} that $W_X$ is generated by $\{ t\in
T\mid X\subseteq H_t\,\}$. It then follows that $X=\Fix(W_X)$, and so the
assignment $X\mapsto W_X$ defines an injection from $L(\CA)$ to the set of
subgroups of $W$. Notice that $W_X$ is again a Coxeter group.

The action of $W$ on $\CA$ induces an action of $W$ on $L(\CA)$. Obviously,
$wW_Xw\inverse= W_{w(X)}$ and so for $X$ and $Y$ in $L(\CA)$, the subgroups
$W_X$ and $W_Y$ are conjugate if and only if $X$ and $Y$ lie in the same
$W$-orbit. Thus, the assignment $X\mapsto W_X$ induces a bijection between
the set of orbits of $W$ on $L(\CA)$ and the set of conjugacy classes of
subgroups $W_X$.

By a \emph{shape of $W$} we mean a $W$-orbit in $L(\CA)$. We denote the set
of shapes of $W$ by $\Lambda$. For example, if $W$ is the symmetric group
$S_n$, then $\Lambda$ is in bijection with the set of partitions of $n$ and
with the set of Young diagrams with $n$ boxes. When $\lambda$ is a shape and
$X$ is a subspace in $\lambda$, we say \emph{$\lambda$ is the shape of
  $W_X$.}

It is shown in \cite[\S6.2] {orlikterao:arrangements} that the assignment
$w\mapsto \Fix(w)$ defines a surjection from $W$ to $L(\CA)$.  Composing
with the map that sends an element $X$ in $L(\CA)$ to its $W$-orbit, we get
a map
\[
\sh\colon W\to \Lambda.
\]
We say $\sh(w)$ is the \emph{shape of $w$}. Thus, $\sh(w)$ is the $W$-orbit
of $\Fix(w)$ in $L(\CA)$. Clearly, $\sh$ is constant on conjugacy classes
and so we can define the \emph{shape} of a conjugacy class to be the shape
of any of its elements.

An element $w$ in $W$, or its conjugacy class, is called \emph{cuspidal} if
$\Fix(w)=\Fix(W)$. For example, if $W$ is the symmetric group $S_n$, then
the conjugacy class consisting of $n$-cycles is the only cuspidal class. In
general, there is more than one cuspidal conjugacy class.  Cuspidal elements
and conjugacy classes are called \emph{elliptic} by some authors.

Suppose that $\lambda$ is a shape, $X$ in $L(\CA)$ has shape $\lambda$, and
$C$ is a conjugacy class in $W$ with shape $\lambda$. If $w$ is in $C$, then
$\Fix(w)$ is in the $W$-orbit of $X$ and so $C\cap W_X$ is a non-empty union
of cuspidal $W_X$-conjugacy classes. Geck and Pfeiffer
\cite[\S3.2]{geckpfeiffer:characters} have shown that in fact $C\cap W_X$ is
a single $W_X$-conjugacy class. It follows that $C\mapsto C\cap W_X$ defines
a bijection between the set of conjugacy classes in $W$ with shape $\lambda$
and the set of cuspidal conjugacy classes in $W_X$.

Fix a set $\{X_\lambda \mid \lambda\in \Lambda\,\}$ of $W$-orbit
representatives in $L(\CA)$. Summarizing the preceding discussion we see
that conjugacy classes in $W$ are parametrized by pairs $(\lambda,
C_\lambda)$, where $\lambda$ is a shape and $C_\lambda$ is a cuspidal
conjugacy class in $W_{X_\lambda}$.

\subsection{The Orlik-Solomon algebra}\label{sec:os}
Next, consider the cohomology ring $H^*(M_W)$. Arnold and Brieskorn (see
\cite[\S1.1]{orlikterao:arrangements}) have computed $H^*(M_W)$. In the
following we use the presentation of this algebra given by Orlik and Solomon
\cite{orliksolomon:combinatorics}.

Recall that the set $T$ of reflections in $W$ parametrizes the hyperplanes
in $\CA$. The \emph{Orlik-Solomon algebra of $W$} is the $\BBC$-algebra, $A=
A(\CA)$, with generators $\{\, a_t\mid t\in T\,\}$ and relations
\begin{itemize}
\item $a_{t_1} a_{t_2}= -a_{t_2} a_{t_1}$ for $t_1$ and $t_2$ in $T$ and
\item $\sum_{i=1}^p (-1)^i a_{t_1} \dotsm \widehat{ a_{t_i}} \dotsm a_{t_p}
  =0$ whenever $\{\, H_{t_1}, \dots, H_{t_p} \,\}$ is a linearly dependent
  subset of $\CA$.
\end{itemize}
The algebra $A$ is a skew-commutative, graded, connected $\BBC$-algebra that
is isomorphic as a graded algebra to $H^*(M_W)$. Let $A^p$ denote the degree
$p$ subspace of $A$. Then
\begin{itemize}
\item $A^0\cong \BBC$, 
\item for $1\leq p\leq |S|$ the subspace $A^p$ is spanned by the set of all
  $a_{t_1} \dotsm a_{t_p}$, where $\codim H_{t_1}\cap \dots\cap H_{t_p}= p$,
  and
\item $A^p=0$ for $p>|S|$.
\end{itemize}
See \cite[\S3.1]{orlikterao:arrangements} for details.

The action of $W$ on $\CA$ extends to an action of $W$ on $A$ as algebra
automorphisms. An element $w$ in $W$ acts on a generator $a_t$ of $A$ by
$wa_t= a_{wtw\inverse}$. With this $W$-action $A$ is isomorphic to
$H^*(M_W)$ as graded $W$-algebras.

Orlik and Solomon \cite{orliksolomon:coxeter} have shown that the normalizer
of $W_X$ in $W$ is the setwise stabilizer of $X$ in $W$, that is,
\[
N_W(W_X)= \{\, w\in W\mid w(X)=X\,\}.
\]
For $X$ in $L(\CA)$, define $A_X$ to be the span of $\{\, a_{t_1} \dotsm
a_{t_r} \mid H_{t_1}\cap \dots\cap H_{t_r} =X\,\}$. Proofs of the following
statements may be found in \cite[Corollary 3.27 and Theorem
6.27]{orlikterao:arrangements}.
\begin{itemize}
\item If $\codim X=p$, then $A_X\subseteq A^p$.
\item There are vector space decompositions $A^p \cong \bigoplus_{\codim
    X=p} A_X$ and $A\cong \bigoplus_{X\in L(\CA)} A_X$.
\item For $w$ in $W$, $wA_X=A_{w(X)}$. Thus, $A_X$ is an $N_W(W_X)$-stable
  subspace of $A$.
\end{itemize}

For a shape $\lambda$ in $\Lambda$, set $A_\lambda= \bigoplus_{X\in \lambda}
A_X$. Suppose $X$ is a fixed subspace in $\lambda$ and that $\codim
X=p$. Then $A_\lambda$ is a $W$-stable subspace of $A^p$ and we have
isomorphisms of $\BBC W$-modules
\[
A_\lambda \cong \Ind_{N_W(W_X)}^W (A_X) \quad\text{and} \qquad A\cong
\bigoplus_{\lambda \in \Lambda} A_\lambda
\]
(see \cite{lehrersolomon:symmetric}).

\subsection{Solomon's descent algebra}\label{descent}
In contrast with the Orlik-Solomon algebra $A$, which is defined for every
complex reflection group, Solomon's descent algebra is defined using the
Coxeter generating set $S$ of $W$ and so has no immediate analog for complex
reflection groups that are not Coxeter groups.

Suppose that $I$ is a subset of $S$. Define
\[
X_I=\bigcap_{s\in I} H_s\qquad \text{and} \qquad W_I =\langle I\rangle.
\]
Then $X_I$ is in $L(\CA)$ and $\codim X_I= |I|$. It follows from Steinberg's
Theorem \cite{steinberg:differential} that $W_I=W_{X_I}$.  Recall that
$\Delta=\{\, \alpha_s\mid s\in S\,\}$ is a base of $\Phi$.  For $I\subseteq
S$ define
\[
\Delta_I=\{\, \alpha_s\mid s\in I\,\}.
\]
Then $X_I$ is the orthogonal complement of the span of $\Delta_I$.

Orlik and Solomon (see \cite[\S6.2] {orlikterao:arrangements}) have shown
that each orbit of $W$ on $L(\CA)$ contains a subspace $X_I$ for some subset
$I$ of $S$. For subsets $I$ and $J$ of $S$ define $I\sim J$ if there is a
$w$ in $W$ with $w(\Delta_I)=\Delta_J$. Then $\sim$ is an equivalence
relation. It is well-known that $W_I$ and $W_J$ are conjugate if and only if
$I\sim J$ (see \cite{solomon:mackey}). It follows that the assignment
$I\mapsto X_I$ induces a bijection between the set of $\sim$-equivalence
classes of subsets of $S$ and the set of shapes of $W$. We use this
bijection to identify shapes of $W$ with subsets of the power set of $S$.

Next, let $\ell$ denote the length function of $W$ determined by the
generating set $S$ and define
\[
W^I=\{\, w\in W\mid \ell(ws)> \ell(w)\,\forall\,s\in I\,\}.
\]
Then $W^I$ is a set of left coset representatives of $W_I$ in $W$. Also,
define
\[
x_I=\sum_{ w\in W^I} w
\]
in the group algebra $\BBC W$. Solomon \cite{solomon:decomposition} has
shown that the span of $\{\, x_I\mid I\subseteq S\,\}$ is in fact a
subalgebra of $\BBC W$. This subalgebra is denoted by $\Sigma(W)$ and is
called the \emph{descent algebra} of $W$. It is not hard to see that $\{\,
x_I\mid I\subseteq S\,\}$ is linearly independent and so $\dim
\Sigma(W)=2^{|S|}$. Notice that $x_S=1$ is the identity in both $\BBC W$ and
its subalgebra $\Sigma(W)$.

Bergeron, Bergeron, Howlett, and Taylor
\cite{bergeronbergeronhowletttaylor:decomposition} have defined a basis
$\{\, e_I\mid I\subseteq S\,\}$ of $\Sigma(W)$ that consists of
quasi-idempotents and is compatible with the set of shapes of $W$. For
$\lambda$ in $\Lambda$ define
\[
S_\lambda=\{\, I\subseteq S\mid X_I\in \lambda\,\} \qquad\text{and} \qquad
e_\lambda= \sum_{I\in S_\lambda} e_I.
\]
Then each $e_\lambda$ is idempotent and $\{\, e_\lambda\mid \lambda\in
\Lambda\,\}$ is a complete set of primitive, orthogonal idempotents in
$\Sigma(W)$. (See \S\ref{prelim} for more details.) In particular,
$\sum_{\lambda\in \Lambda} e_\lambda=1$ in $\BBC W$.

Define $E_\lambda= e_\lambda \BBC W$. In \S\ref{prelim} we show that $e_I
\BBC W_I$ affords an action of $N_W(W_I)$ and that if $I$ is in $S_\lambda$,
then $E_\lambda$ is induced from $e_I \BBC W_I$. Thus, in analogy with the
decomposition in \S\ref{sec:os} of the Orlik-Solomon algebra $A$, we have
isomorphisms of $\BBC W$-modules
\[
E_\lambda \cong \Ind_{N_W(W_I)}^W (e_I \BBC W_I) \qquad \text{and}\qquad
\BBC W\cong \bigoplus_{\lambda\in \Lambda} E_\lambda.
\]

\subsection{Centralizers and complements}\label{central}
The last ingredient we need in order to state the conjecture is a certain
set of characters of centralizers of elements of $W$. These characters,
together with the sign character, should quantify the difference between the
representation of $W$ on $H^p(M_W)$ and a subrepresentation of the regular
representation. They naturally arise in work of Howlett and Lehrer
\cite{howlettlehrer:duality} and in recent results of Konvalinka, Pfeiffer,
and R\"over \cite{konvalinkapfeifferroever:centralizers} that describe the
structure of the centralizer of an element in $W$.

Suppose that $I$ is a subset of $S$ and $C$ is a conjugacy class in $W$ such
that $C\cap W_I$ is a cuspidal conjugacy class in $W_I$. Howlett
\cite{howlett:normalizers} has shown that $W_I$ has a complement, $N_I$, in
$N_W(W_I)$. Moreover, it is shown in
\cite{konvalinkapfeifferroever:centralizers} that if $c$ is in $C\cap W_I$,
then $Z_W(c)\subseteq N_W(W_I)$ and $Z_{W}(c) W_I=N_W(W_I)$. It follows that
\begin{enumerate}
\item for any $X$ in $L(\CA)$, $W_X$ has a complement, say $N_X$, in
  $N_W(W_X)$ such that
  \[
  N_{W}(W_X)\cong W_X\rtimes N_X,
  \]
  and
\item for $c$ in $W_X$ cuspidal, $Z_W(c)\subseteq N_W(W_X)$ and
  $Z_W(c)/ Z_{W_X}(c) \cong N_X$.
\end{enumerate}
 
Recall $N_W(W_X)$ is the setwise stabilizer of $X$ in $W$. Define
$\alpha_X\colon N_W(W_X)\to \BBC$ by $\alpha_X(n)= \det n|_X$.  For $c$ in
$W$ with $X=\Fix(c)$, define $\alpha_c= \alpha_X|_{Z_W(c)}$.  Then
\[
\alpha_c(z)= \det z|_{\Fix(c)}
\]
for $z$ in $Z_W(c)$.

\subsection{Relating the Orlik-Solomon algebra and the descent algebra}
We now have all the concepts we need in order to state the conjecture.

Let $\epsilon$ denote the sign character of $W$. For $c$ in $W$, define
$X_c=\Fix(c)$, $\rk(c)= \codim X_c$, and $W_c= W_{X_c}$. Notice that $c$ is
a cuspidal element in $W_c$.

Associated with each $\lambda$ in $\Lambda$ we have the $W$-stable subspace
$A_\lambda$ of the Orlik-Solomon algebra $A$, the right ideal $E_\lambda$ in
$\BBC W$, and the set of conjugacy classes with shape $\lambda$. We
conjecture that $A_\lambda$ and $E_\lambda$ are related to the set of
conjugacy classes with shape $\lambda$ as follows.

\begin{conjecture}\label{conj}
  Choose a set $\CC$ of conjugacy class representatives in $W$. For
  $\lambda$ in $\Lambda$, set $\CC_\lambda=\{\, c\in \CC\mid
  \sh(c)=\lambda\,\}$. Then, for each conjugacy class representative $c$ in
  $\CC$, there is a linear character $\varphi_c$ of $Z_W(c)$, such that for
  every $\lambda$ in $\Lambda$,
  \begin{enumerate}
  \item[\sl (a)] the character of $E_\lambda$ is
    \[
    \sum_{c\in \CC_\lambda} \Ind_{Z_W(c)}^W(\varphi_c)
    \]
    and
  \item[\sl (b)] the character of $A_\lambda$ is
    \[
    \sum_{c\in \CC_\lambda} \Ind_{Z_W(c)}^W(\epsilon_c \alpha_c \varphi_c )
    \]
    where $\epsilon_c$ denotes the restriction of $\epsilon$ to $Z_W(c)$.
  \end{enumerate}
  In particular,
  \[
  H^p(M_W)\cong A^p\cong \bigoplus_{\rk(c)=p} \Ind_{Z_W(c)}^W (\epsilon_c
  \alpha_c \varphi_c)
  \]
  for $0\leq p\leq |S|$,
  \[
  \BBC W\cong \bigoplus_{c\in \CC} \Ind_{Z_W(c)}^W( \varphi_c),\quad
  \text{and} \quad H^*(M_W)\cong \bigoplus_{c\in \CC}
  \Ind_{Z_W(c)}^W(\epsilon_c \alpha_c \varphi_c).
  \]
\end{conjecture}

As stated in the introduction, we prove Conjecture \ref{conj} for symmetric
groups in \S\ref{lambda=n} and \S\ref{arbitrary}. The conjecture is known to
be true for all Coxeter groups with rank up to six
(\cite{douglasspfeifferroehrle:inductive},
\cite{bishopdouglasspfeifferroehrle:computations},
\cite{bishopdouglasspfeifferroehrle:computationsII}).

We in fact prove more than is stated in the conjecture. First, we show that
the character $\varphi_c$ of $Z_W(c)$ may be chosen to be a trivial
extension of a character of $Z_{W_c}(c)$. Second, we construct explicit
$\BBC W$-module isomorphisms
\[
E_\lambda \cong \Ind_{Z_W(c)}^W(\varphi_c) \quad\text{and} \quad A_\lambda
\cong \Ind_{Z_W(c)}^W(\epsilon_c \alpha_c \varphi_c ).
\]

In \S\ref{sec:rel} we extend the constructions in~\S\ref{lambda=n}
and~\S\ref{arbitrary} and show that if $W$ is any finite Coxeter group,
$\lambda$ is in $\Lambda$, $c$ is in $W$ with $\sh(c)=\lambda$, and the
irreducible components of $W_c$ are all of type $A$, then the character
$\varphi_c$ of $Z_{W_c}(c)$ constructed in~\S\ref{arbitrary} extends to a
character $\widetilde{\varphi}_c$ of $Z_W(c)$. Moreover, we construct
explicit $\BBC W$-module isomorphisms
\[
E_\lambda \cong \Ind_{Z_W(c)}^W(\widetilde{\varphi}_c) \quad\text{and} \quad
A_\lambda \cong \Ind_{Z_W(c)}^W(\epsilon_c \alpha_c \widetilde{\varphi}_c ).
\]
Note that with the given assumptions we have $|\CC_\lambda|=1$, and so the
sums in Conjecture~\ref{conj} (a) and (b) reduce to a single summand. Also,
in contrast with the case when the ambient group $W$ is a symmetric group
and $\widetilde{ \varphi_c}$ is the trivial extension, in the general case,
$\widetilde{\varphi}_c$ may not be the trivial extension of $\varphi_c$.

A direct proof of the conjecture involves finding suitable linear characters
$\varphi_c$ of $Z_{W}(c)$. The computations in
\cite{bishopdouglasspfeifferroehrle:computations} and
\cite{bishopdouglasspfeifferroehrle:computationsII}, as well as calculations
in type $B$, show that some natural guesses about the characters $\varphi_c$
are not true. For example, $Z_W(c)$ acts on the eigenspaces of $c$ in $V$
and so the powers of the determinant character of $Z_W(c)$ acting on an
eigenspace of $c$ are linear characters of $Z_W(c)$. An example in
\cite[\S4]{bishopdouglasspfeifferroehrle:computations} shows that it is not
always possible to choose $\varphi_c$ to be one of these characters for any
eigenspace.

On the other hand, suppose that $c$ is an involution and $X_c$ is the
$-1$-eigenspace of $c$ in $V$. Then $Z_W(c)=N_W(W_{X_c})\cong W_{X_c}
\rtimes N_{X_c}$. Let $\epsilon_{X_c}$ be the sign character of
$W_{X_c}$. Then in all the cases that have been computed so far, it turns
out that $\varphi_c$ may be chosen to be the ``trivial'' extension of
$\epsilon_{W_c}$ to $Z_W(c)$ in the sense that $\varphi_c(wn)=
\epsilon_{X_c}(w)$ for $w$ in $W_{X_c}$ and $n$ in $N_{X_c}$. The
representations $\Ind_{Z_W(c)}^W \varphi_c$ play a role in understanding the
characters of finite reductive groups (\cite{kottwitz:involutions},
\cite{geckmalle:frobenius}), and the corresponding representations of the
Iwahori-Hecke algebra of $W$ in \cite{lusztigvogan:hecke} play a role in the
representation theory of complex reductive groups. The fact that these
representations seem to be closely related to the representation of $W$ on
$H^*(M_W)$ is quite mysterious.

\section{BBHT idempotents}
\label{prelim} 

In this section we collect several preliminary results about the descent
algebra of $W$ and the BBHT idempotents.

For subsets $I$, $J$, and $K$ of $S$ define
\[
W^{IJ}= \left(W^I\right)\inverse \cap W^J \qquad \text{and} \qquad W^{IJK}=
\{\, w\in W^{IJ}\mid w\inverse(\Delta_I) \cap \Delta_J=\Delta_K\,\}.
\]
Then $W^{IJ}$ is the set of minimal length $(W_I,W_J)$-double coset
representatives in $W$. Solomon \cite{solomon:decomposition} has shown that
$x_Ix_J=\sum_K a_{IJK} x_K$ where $a_{IJK}=|W^{IJK}|$.

Let $2^S$ denote the power set of $S$ and fix a function $\sigma\colon
2^S\to \BBR_{>0}$. Notice that $W^K=\{\, w\in W\mid w(\Delta_K)\subseteq
\Phi^+\,\}$ where $\Phi^+$ is the positive system determined by $\Delta$.
Thus, for $J\subseteq K$ and $w$ in $W^K$ we have $w(\Delta_J)\subseteq
\Phi^+$. For subsets $J$ and $K$ of $S$ define
\[
m_{JK}^\sigma= \sum_{\substack{w\in W^K\\ w(\Delta_J)\subseteq \Delta}}
\sigma({}^{w}J) 
\quad\text{if}\quad J\subseteq K \quad\text{and}\quad m_{JK}^\sigma=0
\quad\text{if}\quad J\not\subseteq K.
\]
Because $\sigma(I)>0$ for all subsets $I$ of $S$, we have
$m_{JJ}^\sigma\ne0$ for all $J$ and so the system of equations
\begin{equation}
  \label{eq:system}
  x_K=\sum_{J\subseteq S} m_{JK}^\sigma e_J^\sigma, \qquad K\subseteq S  
\end{equation}
can be solved uniquely for $\{\, e_J^\sigma \mid J\subseteq S\,\}$.  Define
$n_{JK}^\sigma$ and $e_K^\sigma$ by
\[
e_K^\sigma=\sum_{J\subseteq S} n_{JK}^\sigma x_J.
\]
Then $e_K^\sigma$ is in $\Sigma(W)$, $n_{KK}^\sigma=
(m_{KK}^\sigma)\inverse$ for all subsets $K$ of $S$, and $n_{JK}^\sigma=0$
when $J\not\subseteq K$.

Bergeron, Bergeron, Howlett, and Taylor \cite[\S7]
{bergeronbergeronhowletttaylor:decomposition} have shown that $e_I^\sigma$
is a quasi-idempotent in $\Sigma(W)$. Precisely, for $\lambda$ in $\Lambda$
define $\sigma(\lambda)=\sum_{I\in S_\lambda} \sigma(I)$. Then
\begin{equation}
  \label{eq:bbht}
  e_I^\sigma e_J^\sigma = \sigma(\lambda)\inverse e_J^\sigma 
\end{equation}
when $I$ and $J$ are in $S_\lambda$.  Thus, if we set
\[
e_\lambda^\sigma= \sum_{I\in S_\lambda} \sigma(I) e_I^\sigma,
\]
it follows from~(\ref{eq:bbht}) that $e_\lambda^\sigma$ is an idempotent in
$\Sigma(W)$ and hence an idempotent in $\BBC W$. We call the
quasi-idempotents $e_I^\sigma$ \emph{BBHT quasi-idempotents} and the
idempotents $e_\lambda^\sigma$ \emph{BBHT idempotents.}

By definition we have $1=x_S= \sum_{J\subseteq S} n_{JS}^\sigma e_J^\sigma$
and so $1= \sum_{\lambda\in \Lambda}e_\lambda^\sigma$ in $\Sigma(W)$ and
$\BBC W$. It follows that $\{\, e_\lambda^\sigma\mid \lambda\in \Lambda\,\}$
is a set of pairwise orthogonal idempotents in $\BBC W$ and that
\begin{equation}
  \label{eq:multIlambda}
  e_\lambda ^\sigma e_I^\sigma= e_I^\sigma \quad\text{and}\quad e_I^\sigma
  e_\lambda^\sigma = \sigma(\lambda)\inverse e_\lambda ^\sigma
\end{equation}
for $I\in S_\lambda$.

In the special case when the function $\sigma(I)=1$ for all subsets $I$ of
$S$ we do not include $\sigma$ in the notation. Thus,
\begin{equation}\label{eqn:sigma=1}
  \begin{gathered}
    m_{JK} = \left|\{\, w\in W^K\mid w(\Delta_J)\subseteq \Delta\,\}
    \right| \text{\ for\ } J\subseteq K, \\
    e_K=\sum_{J\subseteq S} n_{JK} x_J,\quad \text{and} \quad
    e_\lambda = \sum_{I\in S_\lambda} e_I.
  \end{gathered}  
\end{equation}

Notice that the quantities $m_{JK}^\sigma$, $e_J^\sigma$, $m_{JK}$, $e_J$,
\dots\ are defined relative to an ambient Coxeter system $(W,S)$.  Below we
also consider the analogous quantities defined relative to a parabolic
subsystem $(W_L, L)$. To help keep things straight, in this section and the
next we use the following conventions:
\begin{itemize}
\item $\sigma$ always denotes a function from $2^S$ to $\BBR_{>0}$.

\item When $\sigma(I)=1$ for all $I\subseteq S$, the BBHT quasi-idempotents
  in $\BBC W$ defined with respect to $\sigma$ are denoted by $e_J$.

\item $\tau$ always denotes a function from $2^L$ to $\BBR_{>0}$ where $L$
  is a subset of $S$.

\item When $\tau(I)=1$ for $I\subseteq L$, the BBHT quasi-idempotents in
  $\BBC W_L$ defined with respect to $\tau$ are denoted by $e_J^L$. Thus,
  $e_J^S=e_J$.
\end{itemize}

For example, $\{S\}$ is a shape of $W$ and $\{L\}$ is a shape of $W_L$.
Then $e_{\{S\}}^\sigma= \sigma(S) e_{S}^\sigma$, $e_{\{S\}}= e_S$, and
$e_{L}^L= e_L^\tau= e_{\{L\}} ^\tau$ when $\tau(I)=1$ for all subsets $I$ of
$L$.

The following lemmas give some translation properties for the quantities
defined above.

\begin{lemma}\label{lem:shift}
  Suppose that $K\subseteq S$ and $d$ is in $W$ with $d\inverse (\Delta_K)
  \subseteq \Delta$. Then
  \begin{enumerate}
  \item[\sl (a)] $x_{K^d}= x_K d$;
  \item[\sl (b)] $m_{I^dJ^d}^\sigma= m_{IJ}^\sigma$ for $I\subseteq J
    \subseteq K$; and
  \item[\sl (c)] $e_{L^d}^\sigma= e_L^\sigma d$ for $L \subseteq K$.
  \end{enumerate}
\end{lemma}

\begin{proof}
  It is shown in \cite[Lemma 2.4]
  {bergeronbergeronhowletttaylor:decomposition} that $W^{K^d}=W^K d$.
  Statement (a) follows immediately.

  Suppose that $I\subseteq J\subseteq K$. Clearly, $\Delta_{L^d}=
  d\inverse(\Delta_L)$ for all $L\subseteq K$ and so
  \[
  m_{I^dJ^d}^\sigma= \sum_{\substack{w\in W^{J^d} \\
      w(\Delta_{I^d})\subseteq \Delta}} \sigma({}^{w}(I^d))
  = \sum_{\substack{wd\inverse\in W^{J} \\
      wd\inverse(\Delta_{I})\subseteq \Delta}} \sigma({}^{wd\inverse}I)
  = \sum_{\substack{y\in W^{J} \\
      y(\Delta_{I})\subseteq \Delta}} \sigma({}^{y}I) =m_{IJ}^\sigma.
  \]
  This proves (b).

  Using (a) and (b) we see that for $J\subseteq K$,
  \[
  \sum_{I} m_{IJ}^\sigma (e_I^\sigma d)= x_{J}d =x_{J^d}= \sum_{I}
  m_{IJ^d}^\sigma e_{I}^\sigma = \sum_{I} m_{{}^{d}IJ}^\sigma e_I^\sigma
  = \sum_{I} m_{IJ}^\sigma e_{I^d}^\sigma.
  \] 
  Thus, $\sum_{I} m_{IJ}^\sigma (e_I^\sigma d) = \sum_{I} m_{IJ}^\sigma
  e_{I^d}^\sigma$. Now fix a subset $L$ of $K$, multiply both sides by
  $n_{JL}^\sigma$, and sum over $J$, to get $e_L^\sigma d=
  e_{L^d}^\sigma$. (Note that $n_{JL}^\sigma=0$ unless $J\subseteq L$.)
  This proves (c).
\end{proof}

Recall that we have fixed a positive, definite, Hermitian form on $V$ such
that $W$ is a subgroup of $\U(V)$, the unitary group of $V$. Define
\[
N(W)=\{\, n\in \U(V) \mid n(\Delta)=\Delta\,\}.
\]
Notice that if $n$ is in $N(W)$, then $nSn\inverse =S$. Thus, $N(W)$ acts on
$S$ and on $2^S$, and $WN(W)\cong W\rtimes N(W)$.

\begin{lemma}\label{lem:n}
  Suppose that $n$ is in $N(W)$ and that $\sigma(I^n)=\sigma(I)$ for all
  $I\subseteq S$. Then $e_{I^n}^\sigma = n\inverse e_I^\sigma n$ for
  $I\subseteq S$. In particular, $n$ centralizes $e_S^\sigma$ in $\BBC W$.
\end{lemma}

\begin{proof}
  Because $n(\Delta)=\Delta$, it follows that $\ell(nwn\inverse)= \ell(w)$
  for all $w$ in $W$. Therefore, $n\inverse W^I n=W^{I^n}$ and hence
  $n\inverse x_In=x_{I^n}$ for all $I\subseteq S$.

  Suppose $I\subseteq J\subseteq S$. Then
  \[
  \sum_{\substack{w\in W^{J^n} \\w(\Delta_{I^n})\subseteq \Delta}}
  \sigma({}^{w}(I^n)) = \sum_{\substack{nwn\inverse\in W^{J}
      \\nwn\inverse(\Delta_{I})\subseteq \Delta}} \sigma(I^{nw\inverse}) =
  \sum_{\substack{y\in W^{J} \\ y(\Delta_{I})\subseteq \Delta}}
  \sigma(({}^{y}I)^n) = \sum_{\substack{y\in W^{J} \\
      y(\Delta_{I})\subseteq \Delta}} \sigma({}^{y}I)
  \]
  and so $m_{I^nJ^n}^\sigma= m_{IJ}^\sigma$. Thus, $n_{I^nJ^n}^\sigma=
  n_{IJ}^\sigma$ and it follows from~(\ref{eq:system}) that $n\inverse
  e_I^\sigma n = e_{I^n}^\sigma$.
\end{proof}

\begin{lemma}\label{lem:w0}
  Let $w_0$ denote the longest element in $W$.
  \begin{enumerate}
  \item[\sl (a)] The element $w_0$ is in $\Sigma(W)$ and
    \[
    w_0= \sum_{L\subseteq S} (-1)^{|L|} e_L = \sum_{\lambda\in \Lambda}
    (-1)^{d_\lambda} e_\lambda,
    \]
    where $d_\lambda= |L|$ for $L$ in $S_\lambda$.
  \item[\sl (b)] For $J\subseteq S$, $w_0 e_J=(-1)^{|J|} e_J$,
    $w_0e_Jw_0=e_{J^{w_0}}$, and $e_Jw_0= (-1)^{|J|} e_{J^{w_0}}$.
  \end{enumerate}
\end{lemma}

\begin{proof}
  For $J\subseteq S$ define $Y^J=\{\, w\in W\mid \{\, s\in S\mid ws>w \,\}=J
  \,\}$. Then $W^J= \amalg_{J\subseteq K} Y^K$ and so $W^J\cap P=
  \amalg_{J\subseteq K} (Y^K \cap P)$ for every subset $P$ of $W$. For
  $P\subseteq W$ define
  \[
  p_J^P= \sum_{w\in W^J \cap P}w, \quad q_K^P= \sum_{w\in Y^K \cap P} w,
  \quad f_J^P= |W^J \cap P|, \quad \text{and}\quad g_K^P=|Y^K \cap P|.
  \] 
  Then $p_J^P= \sum_{J\subseteq K} q_K^P$ and $f_J^P= \sum_{J\subseteq K}
  g_K^P$, and so by M\"obius inversion,
  \begin{equation}
    \label{eq:yja}
    q_J^P= \sum_{J\subseteq K} (-1)^{|K|-|J|} p_K^P \quad\text{and}\quad
    g_J^P= \sum_{J\subseteq K} (-1)^{|K|-|J|} f_K^P .
  \end{equation}

  Taking $J=\emptyset$ and $P=W$ in~(\ref{eq:yja}), we have
  $Y^\emptyset=\{w_0\}$ and so
  \begin{equation}
    \label{eq:w0}
    w_0=q_\emptyset^W = \sum_K (-1)^{|K|} p_K^W= \sum_K (-1)^{|K|} x_K.    
  \end{equation}

  For $J\subseteq S$, set $W^J_+=\{\, w\in W \mid w(\Delta_J) \subseteq
  \Delta\,\}$. Then for $J\subseteq K$, $m_{JK}= |W^K \cap W^J_+|$.  Taking
  $P=W^J_+$ in~(\ref{eq:yja}), we have $Y^J\cap W^J_+= \{ w_0w_J\}$, where
  $w_J$ is the longest element in $W_J$, and so $g_J^{P}=1$ and $f_K^{P} =
  m_{JK}$. Therefore,
  \begin{equation}
    \label{eq:mjk}
    (-1)^{|J|} = \sum_{J\subseteq K} (-1)^{|K|} m_{JK}.    
  \end{equation}
 
  Using~(\ref{eq:w0}) and ~(\ref{eq:mjk}), we have
  \begin{multline*}
    w_0= \sum_K (-1)^{|K|} x_K =\sum_K (-1)^{|K|} \sum_{J} m_{JK} e_J
    = \sum_J \left(\sum_{J\subseteq K} (-1)^{|K|} m_{JK} \right) e_J\\
    = \sum_J (-1)^{|J|} e_J = \sum_{\lambda\in \Lambda} (-1)^{d_\lambda}
    e_\lambda.
  \end{multline*}

  Now suppose that $J\subseteq S$ and $\mu \in \Lambda$ are such that $J\in
  S_\mu$. By (\ref{eq:multIlambda}), we have $e_\mu e_J=e_J$ and so
  \[
  w_0e_J= \sum_{\lambda\in \Lambda} (-1)^{d_\lambda} e_\lambda e_\mu e_J =
  (-1)^{d_\mu} e_\mu e_J = (-1)^{|J|} e_J.
  \]
  For subsets $J$ and $K$ of $S$ we have $m_{JK}= m_{J^{w_0}
    K^{w_0}}$. Thus, $n_{JK}= n_{J^{w_0} K^{w_0}}$ and it follows
  from~(\ref{eq:system}) that $w_0e_Jw_0=e_{J^{w_0}}$. Finally,
  \[
  e_Jw_0= w_0 (w_0e_Jw_0)= (-1)^{|J^{w_0}|} e_{J^{w_0}} =(-1)^{|J|}
  e_{J^{w_0}}.
  \]
  This completes the proof of the lemma.
\end{proof}

For a subset $L$ of $S$, the pair $(W_L,L)$ is a Coxeter system. Because
$W^L$ is a complete set of left coset representatives of $W_L$ in $W$, left
multiplication by $x_L$ defines an embedding of $\BBC W_L$ into $\BBC
W$. For $I\subseteq L$ define $W_L^I= W_L\cap W^I$ and $x^L_I=\sum_{ w\in
  W_L^I} w$. Then $\{\, x^L_I \mid I\subseteq L\,\}$ is a basis of
$\Sigma(W_L)$. It is well-known and easy to prove that $W^L W_L^I= W^I$, and
so $x_Lx_I^L=x_I$. If $n$ is in $N_L$, then $n(\Delta_L)= \Delta_L$ and so
by Lemma \ref{lem:shift}(a) we have $x_Ln=x_L$.  It follows that $x_L \BBC
W_L$ is stable under right multiplication by elements of $N_W(W_L)$.

It is straightforward to check that for $a$ in $\BBC W_L$ and $wn$ in
$W_LN_L$, the assignment $(a, wn) \mapsto a\cdot wn =n\inverse awn$ defines
a right action of the group $N_W(W_L)=W_LN_L$ on $\BBC W_L$. Then
\[
x_L (a\cdot wn) =x_Ln\inverse awn =x_Lawn
\]
and so left multiplication by $x_L$ defines an $N_W(W_L)$-equivariant
embedding of $\BBC W_L$ into $\BBC W$. For later reference we record this
fact in the following lemma.

\begin{lemma}\label{lem:Nequiv}
  Suppose that $L$ is a subset of $S$. Then $N_W(W_L)$ acts on $\BBC W_L$ by
  $a\cdot wn =n\inverse awn$, for $a \in \BBC W_L$, $w$ in $W_L$, and $n \in
  N_L$, and left multiplication by $x_L$ defines an $N_W(W_L)$-equivariant
  embedding of $\BBC W_L$ into $\BBC W$.
\end{lemma}

Recall that the quasi-idempotents $e_I^\sigma$ are defined relative to the
ambient set $S$ and the function $\sigma$. Define
\[ 
{\sigma_L}\colon 2^L\to \BBR_{>0}\qquad \text{by}\qquad{\sigma_L}(I)=
m_{IL}^\sigma.
\]
Then for $J\subseteq L$ we have the quasi-idempotent $e_J^{\sigma_L}=
\sum_{I\subseteq L} n_{IJ}^{\sigma_L} x_I^L$ in $\BBC W_L$ defined relative
to the set $L$ and the function ${\sigma_L}$.

\begin{lemma}\label{lem:restrict}
  Suppose that $I$, $J$, and $L$ are subsets of $S$ with $I, J\subseteq
  L$. Then
  \begin{enumerate}
  \item[\sl(a)] $m_{IJ}^{\sigma_L}= m_{IJ}^\sigma$ and $n_{IJ}^{\sigma_L}=
    n_{IJ}^\sigma$;
  \item[\sl(b)] $x_Le_J^{\sigma_L} = e_J^\sigma$; and
  \item[\sl(c)] $n\inverse e_J^{\sigma_L} n= e_{J^n}^{\sigma_L}$ for $n$ in
    $N_L$.
  \end{enumerate}
\end{lemma}

\begin{proof}
  It is shown in \cite[Theorem 7.5]
  {bergeronbergeronhowletttaylor:decomposition} that $m_{IJ}^{\sigma_L}=
  m_{IJ}^\sigma$. It then follows from the definitions that $n_{IJ}^
  {\sigma_L} = n_{IJ}^\sigma$. This proves (a). Now
  \[
  x_Le_J^{\sigma_L} = x_L\sum_{\substack{ I\\ I\subseteq J}}
  n_{IJ}^{\sigma_L} x_I^L = \sum_{\substack{ I\\ I\subseteq J}}
  n_{IJ}^{\sigma_L} x_L x_I^L = \sum_{\substack{ I\\ I\subseteq J}}
  n_{IJ}^\sigma x_I= e_J^\sigma
  \]
  and so (b) holds.

  Suppose $n$ is in $N_L$. Then using (b) and Lemma \ref{lem:shift} we have
  \[
  e_J^\sigma n= x_Le_J^{\sigma_L} n= (x_Ln)(n\inverse e_J^{\sigma_L} n) =
  x_L (n\inverse e_J^{\sigma_L} n)
  \]
  and
  \[
  e_J^\sigma n= e_{J^n}^\sigma = x_{L^n} e_{J^n}^{\sigma_L} = x_{L} e_{
    J^n}^{\sigma_L}.
  \]
  Now $n\inverse e_J^{\sigma_L} n$ and $e_{J^n}^{\sigma_L}$ are both in
  $\BBC W_L$ and $x_L (n\inverse e_J^{\sigma_L} n)= x_{L}
  e_{J^n}^{\sigma_L}$, so we conclude from Lemma \ref{lem:Nequiv} that
  $n\inverse e_J^{\sigma_L} n= e_{J^n}^{\sigma_L}$. This proves (c).
\end{proof}

We conclude this section with a description of the quasi-idempotents
$e_K^\sigma$ in the case when $W$ is reducible.

Suppose that $W$ is reducible, say $W\cong W_1\times W_2$. Then $S=
S_1\amalg S_2$ is the disjoint union of $S_1$ and $S_2$ where $W_1= \langle
S_1 \rangle$ and $W_2= \langle S_2 \rangle$, and $\Delta=\Delta_1\amalg
\Delta_2$, where $\Delta_i=\{\, \alpha\mid s_\alpha\in S_i\,\}$ for $i=1,2$.

Suppose $K$ is a subset of $S$. Then $K=K_1\amalg K_2$ where $K_i=K\cap S_i$
for $i=1,2$. Every element in $W$ has a unique expression as a product
$w_1w_2$ with $w_1$ in $W_1$ and $w_2$ in $W_2$. Then $w_1w_2$ is in $W_K$
if and only if $w_1$ is in $W_{K_1}$ and $w_2$ is in $W_{K_2}$. Moreover,
$w_1w_2$ is in $W^K$ if and only if $w_1$ is in $W^{K_1}$ and $w_2$ is in
$W^{K_2}$. It follows that $x_K=x_{K_1} x_{K_2}$.  If $J$ is a subset of
$K$, then $\Delta_J= \Delta_{J_1} \amalg \Delta_{J_2}$ and
$w_1w_2(\Delta_J)\subseteq \Delta$ if and only if $w_1(\Delta_{J_1})
\subseteq \Delta_1$ and $w_2(\Delta_{J_2})\subseteq \Delta_2$.  Similarly,
${}^{w_1w_2} (J_1\amalg J_2) = {}^{w_1}{J_1} \amalg {}^{w_2}{J_2}$.

Now suppose that $\sigma\colon 2^S\to \BBR_{>0}$ has the property that
$\sigma(J_1\amalg J_2)= \sigma(J_1) \sigma(J_2)$ for $J_1\subseteq S_1$ and
$J_2\subseteq S_2$. Then for $J\subseteq K\subseteq S$ we have
\begin{align*}
  m_{JK}^\sigma &= \sum_{\substack{w\in W^K\\ w(\Delta_J)\subseteq
      \Delta}} \sigma({}^{w}J) \\
  &= \sum_{\substack{w_1\in W^{K_1},\, w_2 \in W^{K_2}\\
      w_1(\Delta_{J_1})\subseteq \Delta_1,\, w_2(\Delta_{J_2})\subseteq
      \Delta_2}}
  \sigma({}^{w_1}J_1 {\textstyle\amalg} {}^{w_2}J_2)  \\
  &= \sum_{\substack{w_1\in W^{K_1}\\ w_1(\Delta_{J_1}) \subseteq \Delta_1}}
  \sum_{\substack{w_2 \in W^{K_2}\\ w_2(\Delta_{J_2}) \subseteq \Delta_2}}
  \sigma({}^{w_1}J_1) \sigma({}^{w_2}J_2)  \\
  &= m_{J_1K_1}^\sigma m_{J_2K_2}^\sigma \\
  &= m_{J_1K_1}^{\sigma_1} m_{J_2K_2}^{\sigma_2}
\end{align*}
where $\sigma_i$ is the restriction of $\sigma$ to $2^{S_i}$ for $i=1,2$.
Conversely, if we are given functions $\sigma_i\colon 2^{S_i}\to \BBR_{>0}$
for $i=1,2$ and define $\sigma\colon 2^S\to \BBR_{>0}$ by $\sigma(K)=
\sigma(K_1) \sigma_2(K_2)$, then $m_{J_1K_1}^{\sigma_1}
m_{J_2K_2}^{\sigma_2} = m_{J_1K_1}^{\sigma} m_{J_2K_2}^{\sigma} =
m_{JK}^{\sigma}$.

With $\sigma$ as above, set $f_J^\sigma=e_{J_1}^\sigma e_{J_2}^{\sigma}$ for
$J= J_1\amalg J_2 \subseteq S$. Then
\[
\sum_{J\subseteq K} m_{JK}^\sigma f_J^\sigma = \sum_{ \substack{
    J_1\subseteq K_1 \\J_2\subseteq K_2}} m_{J_1K_1}^\sigma
m_{J_2K_2}^\sigma e_{K_1}^\sigma e_{K_2}^{\sigma} = x_{K_1} x_{K_2} = x_K
\]
and so $f_J^\sigma= e_J^\sigma$. This proves the following proposition.

\begin{proposition}\label{pro:red}
  Suppose $W\cong W_1\times W_2$ is reducible and $S=S_1\amalg S_2$ where
  $W_1=\langle S_1 \rangle$ and $W_2=\langle S_2 \rangle$. Suppose in
  addition that $\sigma\colon 2^S\to \BBR_{>0}$ has the property that for
  $J\subseteq S$, $\sigma(J)= \sigma(J\cap S_1) \sigma(J\cap S_2)$.  Then
  \[
  e_J^\sigma= e_{J\cap S_1}^\sigma e_{J\cap S_2}^{\sigma} \quad
  \text{and}\quad e_J= e_{J\cap S_1} e_{J\cap S_2}.
  \]
\end{proposition}

\section{$E_\lambda$ is an induced representation} \label{induced}

Suppose $\lambda$ is in $\Lambda$ and $\sigma\colon 2^S\to
\BBR_{>0}$. Define
\[
E_\lambda^\sigma= e_\lambda^\sigma \BBC W
\]
to be the right ideal in $\BBC W$ generated by $e_\lambda^\sigma$.
Similarly, using the notation in~(\ref{eqn:sigma=1}), define
\[
E_\lambda= e_\lambda \BBC W.
\]
We have seen in \S\ref{sec:os} that $A_\lambda \cong \Ind_{N_W(W_X)}^W
(A_X)$ for $X$ in $\lambda$. In this section we show that $E_\lambda^\sigma$
has a similar description as an induced representation, and we analyze how
$E_\lambda^\sigma$ depends on the choice of $\sigma$. In particular, it is shown
in Corollary \ref{cor:conc} that for $L$ in $S_\lambda$,
\[
E_\lambda \cong \Ind_{N_W (W_L)}^W \left( e_L \BBC W_L \right) \cong
\Ind_{N_W (W_L)}^W \left( e_L^L \BBC W_L \right).
\]

We begin with a lemma which follows immediately from~(\ref{eq:multIlambda}).

\begin{lemma}\label{lem:lambda}
  Suppose that $\lambda$ is in $\Lambda$ and $I$ is in $S_\lambda$. Then
  $E_\lambda^\sigma= e_I^\sigma \BBC W$.
\end{lemma}

The next proposition shows that up to isomorphism, $E_\lambda ^\sigma$ does
not depend on $\sigma$.

\begin{proposition}\label{pro:elambda}
  Suppose that $\lambda$ is in $\Lambda$ and that $\sigma$ and $\sigma_1$
  are functions from $2^S$ to $\BBR_{>0}$. Then there is a unit $u$ in
  $\Sigma(W)$ such that left multiplication by $u$ defines an isomorphism of
  right $\BBC W$-modules $E_\lambda^\sigma \cong E_\lambda^{\sigma_1}$.
\end{proposition}

\begin{proof}
  Let $\rad (\Sigma(W))$ denote the Jacobson radical of $\Sigma(W)$ and let
  $\theta$ denote the natural projection from $\Sigma(W)$ to $\Sigma(W)/
  \rad(\Sigma(W))$. Bergeron, Bergeron, Howlett, and Taylor
  \cite[\S7]{bergeronbergeronhowletttaylor:decomposition} have shown that
  $\theta(e_\lambda^\sigma)= \theta(e_\lambda^{\sigma_1})$ is a primitive
  idempotent in $\Sigma(W)/ \rad(\Sigma(W))$. Thus, it follows from
  \cite[Theorem 3.1]{thevenaz:G-algebras} that there is a unit $u$ in
  $1+\rad(\Sigma(W))$ such that $ue_\lambda ^\sigma= e_\lambda ^{\sigma_1}
  u$.  Then left multiplication by $u$ defines an isomorphism of right $\BBC
  W$-modules $e_\lambda^\sigma \BBC W \cong e_\lambda^{\sigma_1} \BBC W$.
\end{proof}

Suppose that $N$ is a subgroup of $N(W)$, so $WN$ is a subgroup of
$N_{\U(V)}(W)$. Then as in Lemma~\ref{lem:Nequiv}, $WN$ acts on $\BBC W$ on
the right by $a\cdot wn =n\inverse awn$ for $a$ in $\BBC W$, $w$ in $W$, and
$n$ in $N$. If $N$ centralizes $e_\lambda ^\sigma$, then clearly
$E_\lambda^\sigma$ is a $WN$-submodule of $\BBC W$. It follows from Lemma
\ref{lem:n} that if $\sigma$ is constant on $N$-orbits in $2^S$, then $N$
centralizes $e_{\{S\}}^\sigma =e_S^\sigma$. More generally, if $\sigma$ is
constant on $N$-orbits in $2^S$ and $S_\lambda$ is $N$-stable, then $N$
centralizes $e_\lambda ^\sigma$.

Let $\Sigma(W)^{N}$ denote the algebra of $N$-invariants in $\Sigma(W)$.

\begin{proposition}\label{pro:xxx}
  Suppose that $\lambda$ is in $\Lambda$ and that $\sigma$ and $\sigma_1$
  are functions from $2^S$ to $\BBR_{>0}$ such that $S_\lambda$ is
  $N$-stable, $N$ centralizes $e_\lambda ^\sigma$, and $\sigma_1$ is
  constant on $N$-orbits in $2^S$.  Then there is a unit $v$ in
  $\Sigma(W)^N$ such that left multiplication by $v$ defines an isomorphism
  of right $\BBC WN$-modules $E_\lambda^\sigma$ and $E_\lambda^{\sigma_1}$.
\end{proposition}

\begin{proof}
  We saw in the proof of Proposition \ref{pro:elambda} that there is a unit
  $u$ in $1+\rad(\Sigma(W))$ such that $ue_\lambda ^\sigma= e_\lambda
  ^{\sigma_1} u$, and we observed in the proof of Lemma \ref{lem:n} that
  $n\inverse x_In= x_{I^n}$ for $I\subseteq S$ and $n$ in $N$. It follows
  that $\Sigma(W)$ and $\rad(\Sigma(W))$ are stable under conjugation by $N$
  and so $v= \frac 1{|N|} \sum_{n\in N} n\inverse un$ is in
  $1+\rad(\Sigma(W))$ and hence is a unit. By assumption, $N$ centralizes
  $e_\lambda ^\sigma$ and $e_\lambda ^{\sigma_1}$ and it follows that
  $ve_\lambda ^\sigma= e_\lambda ^{\sigma_1} v$. Therefore, left
  multiplication by $v$ defines a $N_{\U(V)}(W)$-module isomorphism
  $e_\lambda^\sigma \BBC W \cong e_\lambda^{\sigma_1} \BBC W$.
\end{proof}

In the next proposition, $\ell_x$ denotes left multiplication by $x$.

\begin{proposition}\label{pro:commutes}
  Suppose $L$ is a subset of $S$, $\sigma\colon 2^S\to \BBR_{>0}$,
  $\tau\colon 2^L\to \BBR_{>0}$, and $\tau$ is constant on $N_L$-orbits.
  Then there is a unit $v$ in $\Sigma(W_L)^{N_L}$ such that if $\mu_v=
  \ell_{x_L} \ell_v \ell_{x_L} \inverse$, then the diagram
  \[
  \xymatrix{ e_L^{\sigma_L} \BBC W_L \ar[r]^{\ell_v} \ar[d]_{\ell_{x_L}}
    & e_L^\tau \BBC W_L \ar[d]^{\ell_{x_L}} \\
    e_L ^\sigma \BBC W_L \ar[r]_{\mu_v} & x_L e_L^\tau \BBC W_L }
  \]
  is a commutative diagram of right $N_W(W_L)$-modules and
  $N_W(W_L)$-isomorphisms.
\end{proposition}

\begin{proof}
  It follows from Lemma \ref{lem:restrict}(c) and Lemma \ref{lem:n} that
  $e_L^{\sigma_L} \BBC W_L$ and $e_L^\tau \BBC W_L$ are $N_W(W_L)$-stable
  right ideals of $\BBC W_L$, where $N_W(W_L)$ acts on $\BBC W_L$ as in
  Lemma \ref{lem:Nequiv}. By Lemma \ref{lem:restrict}(b) we have
  $e_L^{\sigma} =x_Le_L ^{\sigma_L}$ and so it follows from Lemma
  \ref{lem:shift}(a) that $e_L ^\sigma \BBC W_L$ and $x_L e_L^\tau \BBC W_L$
  are stable under right multiplication by $N_W(W_L)$. It now follows from
  Lemma \ref{lem:Nequiv} that the vertical maps are isomorphisms of
  $N_W(W_L)$-modules.

  The hypotheses of Proposition \ref{pro:xxx} are satisfied with $L$,
  $\{L\}$, and $\sigma_L$ in place of $S$, $S_\lambda$, and $\sigma$. Thus,
  there is a unit $v$ in $\Sigma(W_L)^{N_L}$ such that $\ell_v\colon
  e_L^{\sigma_L} \BBC W_L \to e_L^\tau \BBC W_L$ is an isomorphism. The
  conclusion of the proposition is now clear, as $\mu_v= \ell_{x_L} \ell_v
  \ell_{x_L} \inverse$.
\end{proof}

The decomposition $1= \sum_{\lambda\in \Lambda}e_\lambda^\sigma$ gives a
decomposition
\[
\BBC W \cong \bigoplus_{\lambda\in \Lambda} E_\lambda^\sigma
\]
of $\BBC W$ into right ideals. It is shown in
\cite[\S7]{bergeronbergeronhowletttaylor:decomposition} that $\dim
E_\lambda^\sigma = |\sh\inverse(\lambda)|$, the number of elements in $W$
with shape $\lambda$. In the next lemma we compute $|\sh\inverse(\lambda)|$
in terms of cuspidal elements in a parabolic subgroup of $W$ with shape
$\lambda$.

\begin{lemma}\label{lem:cuspidal}
  Suppose that $\lambda$ is a shape in $\Lambda$, $X$ is a subspace in
  $\lambda$, and $C$ is a conjugacy class in $W$ with shape $\lambda$.  Then
  \begin{enumerate}
  \item[\sl (a)] $|C|= |W: N_W(W_X)| \, |C\cap W_X|$ and
  \item[\sl (b)] $|\sh\inverse(\lambda)|= |W: N_W(W_X)| \,
    |\sh\inverse(\lambda) \cap W_X|$.
  \end{enumerate}
\end{lemma}

\begin{proof}
  Notice that $C\cap W_X$ is a cuspidal conjugacy class in $W_X$. Thus, it
  follows from (1) and (2) in \S\ref{central} that $|N_W(W_X): Z_W(c)|=
  |W_X: Z_{W_X}(c)|$ for $c$ in $C\cap W_X$. Therefore
  \[
  |C|= |W:N_W(W_X)| \,|N_W(W_X): Z_{W}(c)| = |W: N_W(W_X)| \, |C\cap W_X|.
  \]
  This proves (a). Statement (b) follows from (a) and the observation that
  $\sh\inverse(\lambda)$ is the union of those conjugacy classes in $W$
  whose intersection with $W_X$ is a cuspidal conjugacy class in $W_X$.
\end{proof}

\begin{corollary}\label{cor:dimA=dimE}
  Suppose $\lambda$ is in $\Lambda$, $J$ is in $S_\lambda$, $\sigma\colon
  2^S\to \BBR_{>0}$, and $\tau\colon 2^J\to \BBR_{>0}$.  Then
  \begin{enumerate}
  \item[\sl (a)] $\dim A_{X_J}= \dim e_J^\sigma \BBC W_J = \dim e_J^\tau
    \BBC W_J =|\sh\inverse(\lambda) \cap W_J|$ and
  \item[\sl (b)] $\dim A_\lambda= \dim E_\lambda ^\sigma= |\sh \inverse(
    \lambda)|$.
  \end{enumerate}
\end{corollary}

\begin{proof}
  It is clear that $\{\, w\in W\mid \Fix(w)=X_J\,\} = \sh\inverse( \lambda)
  \cap W_J$ is the set of all cuspidal elements in $W_J$, and it is shown in
  \cite[Proposition 2.4]{douglass:cohomology} that $\dim A_{X_J}=| \{\, w\in
  W\mid \Fix(w)=X_J\,\}|$. Therefore, $\dim A_{X_J}= |\sh\inverse(\lambda)
  \cap W_J|$, and so it follows from Lemma \ref{lem:cuspidal}(b) that $\dim
  A_\lambda= |\sh\inverse (\lambda)|$.  It is shown in \cite[Theorem
  7.15]{bergeronbergeronhowletttaylor:decomposition} that $\dim
  E_\lambda^\sigma = |\sh\inverse (\lambda)|$. This proves (b).

  It remains to show that $\dim e_J^\sigma \BBC W_J = \dim e_J^\tau \BBC W_J
  = |\sh\inverse(\lambda) \cap W_J|$. Using Lemma \ref{lem:restrict}(b),
  Lemma \ref{lem:Nequiv}, Proposition \ref{pro:elambda}, and (b) applied to
  the shape $\{J\}$ of $W_J$, we have
  \[
  \dim e_J^\sigma \BBC W_J= \dim x_Je_J^{\sigma_J} \BBC W_J =\dim
  e_J^{\sigma_J} \BBC W_J =\dim e_J^{\tau} \BBC W_J = |\sh\inverse(\lambda)
  \cap W_J|,
  \]
  as desired.
\end{proof}

The next proposition and its corollary are the main results in this section.

\begin{proposition}\label{pro:induct}
  Suppose that $\sigma\colon 2^S\to \BBR_{>0}$, $\lambda$ is in $\Lambda$,
  $X$ is in $\lambda$, and $L$ is in $S_\lambda$.
  \begin{itemize}
  \item[\sl (a)] $N_W(W_L)$ acts on $e_L^\sigma \BBC W_L$ by right
    multiplication and
    \[
    E_\lambda^\sigma \cong \Ind_{N_W(W_L)}^W\left( e_L^\sigma \BBC W_L
    \right).
    \]
  \item[\sl (b)] $N_W(W_X)$ acts on $A_X$ and 
    \[
    A_\lambda \cong \Ind_{N_W(W_X)}^W (A_X).
    \]
  \end{itemize}
\end{proposition}

\begin{proof}
  Statement (b) is proved in \cite[\S2]{lehrersolomon:symmetric}. We prove
  (a).

  It was shown in Proposition \ref{pro:commutes} that $e_L^\sigma \BBC W_L$
  is stable under right multiplication by $N_W(W_L)$. By Lemma
  \ref{lem:lambda}, $E_\lambda^\sigma = e_L^\sigma\BBC W$. Therefore, to
  prove that $E_\lambda ^\sigma \cong \Ind_{N_W(W_L)}^W\left( e_L ^\sigma
    \BBC W_L \right)$ it is enough to show that the multiplication map
  $e_L^\sigma \BBC W_L \otimes_{\BBC N_W(W_L)} \BBC W \to E_\lambda ^\sigma$
  is a bijection. This map is obviously a surjection. Using Lemma
  \ref{lem:cuspidal} and Corollary \ref{cor:dimA=dimE}, we have
  \begin{align*}
    \dim E_\lambda^\sigma &= |\sh\inverse(\lambda)| \\
    &= |W: N_W(W_L)| \, |\sh\inverse(\lambda) \cap W_L|\\
    &= |W: N_W(W_L)| \, \dim e_L^\sigma \BBC W_L \\
    &= \dim e_L^\sigma \BBC W_L \otimes_{\BBC N_W(W_L)} \BBC W
  \end{align*}
  and so the multiplication map is indeed a bijection.
\end{proof}

We saw in Proposition \ref{pro:commutes} that $e_L ^\sigma \BBC W_L$,
$e_L^{\sigma_L} \BBC W_L$, $e_L^\tau \BBC W_L$, and $x_L e_L^\tau \BBC W_L$
all afford equivalent representations of $N_W(W_L)$ when $\tau\colon 2^L\to
\BBR_{>0}$ is chosen only subject to the restriction that it is constant on
$N_L$-orbits. Thus, Proposition \ref{pro:induct}(a) implies the following
corollary.

\begin{corollary}\label{cor:conc}
  Suppose $\lambda$ is in $\Lambda$, $L$ is in $S_\lambda$, $\sigma\colon
  2^S\to \BBR_{>0}$, and $\tau\colon 2^L\to \BBR_{>0}$ is constant on
  $N_L$-orbits. Then
  \[
  E_\lambda^\sigma \cong \Ind_{N_W (W_L)}^W \left( e_L^{\sigma_L} \BBC W_L
  \right) \cong \Ind_{N_W (W_L)}^W \left( e_L^\tau \BBC W_L \right) \cong
  \Ind_{N_W (W_L)}^W \left( x_Le_L^\tau \BBC W_L \right).
  \]
  In particular, if $\sigma(I)=1$ and $\tau(J)=1$ for all $I\subseteq S$ and
  $J\subseteq L$, then
  \[
  E_\lambda \cong \Ind_{N_W (W_L)}^W \left( e_L^L \BBC W_L \right).
  \]
\end{corollary}

\section{Symmetric groups: $\lambda=(n)$}\label{lambda=n}

In this section and the next we prove Conjecture \ref{conj} for symmetric
groups. In these two sections we take $W$ to be the symmetric group on $n$
letters with $n\geq 2$ and we identify $W$ with the subgroup of
$\GL_n(\BBC)$ that acts on the basis $\{v_1, v_2, \dots, v_n\}$ as
permutations. Here, $v_i$ is the column vector whose $j\th$ entry is $0$ for
$j\ne i$ and $1$ for $j=i$. For $1\leq i\leq n-1$ let $s_i$ denote the
matrix in $W$ that interchanges $v_i$ and $v_{i+1}$ and fixes $v_j$ for
$j\ne i, i+1$.  Then $S=\{s_1, s_2, \dots, s_{n-1}\}$ is a Coxeter
generating set for $W$.

By a \emph{partition of $n$} we mean a non-increasing finite sequence of
positive integers whose sum is $n$. Say $\lambda = (\lambda_1, \dots,
\lambda_p)$ is a partition of $n$. Then $\lambda_1\geq \dots \geq \lambda_p
>0$ and $\sum_{k=1}^p \lambda_k=n$. The integers $\lambda_i$ are called the
\emph{parts of $\lambda$.}

It is well-known that for $W=S_n$ we may identify $\Lambda$ with the set of
partitions of $n$. We make this identification precise as follows.  Suppose
that $\lambda$ is a partition of $n$ with $p$ parts. Define partial sums
$\tau_i$ for $i=0, 1, \dots, p$ by $\tau_0=0$ and $\tau_i= \lambda_1+ \dots
+\lambda_i$ for $1\leq i\leq p$. Define
\[
I_\lambda= S\setminus \{s_{\tau_1}, s_{\tau_2}, \dots, s_{\tau_{p-1}} \}
\qquad \text{and}\qquad W_\lambda=\langle I_\lambda \rangle.
\]
Then $W_\lambda$ is isomorphic to the product of symmetric groups
$S_{\lambda_1}\times \dots \times S_{\lambda_p}$, where the factor
$S_{\lambda_i}$ acts on the subset $\{v_{\tau_{i-1}+1}, v_{\tau_{i-1}+2},
\dots, v_{\tau_{i}} \}$ of $\{v_1, v_2, \dots, v_n\}$.  Next, define
\[
X_\lambda= \Fix(W_\lambda).
\]
Then $X_\lambda$ is in $L(\CA)$ and $W_{X_\lambda}= W_\lambda$.  We have
seen in Proposition \ref{pro:induct} that
\[
E_\lambda\cong \Ind_{N_W(W_\lambda)} ^W (e_{I_\lambda} \BBC W_\lambda)
\quad\text{and} \quad A_\lambda \cong \Ind_{N_W(W_\lambda)} ^W
(A_{X_\lambda}).
\]
It is well-known and straightforward to check that $\{\, X_\lambda \mid
\text{$\lambda$ is a partition of $n$} \,\}$ is a complete set of orbit
representatives for the action of $W$ on $L(\CA)$ and that $\{\, I_\lambda
\mid \text{$\lambda$ is a partition of $n$} \,\}$ is a complete set of
representatives for the $\sim$-equivalence classes of subsets of $S$.

Notice that in the extreme case when all parts of $\lambda$ are equal to $1$
we have $I_\lambda=\emptyset$ and $W_\lambda=W_\emptyset= \{1\}$. At the
other extreme, when $\lambda=(n)$, we have $I_\lambda=S$ and
$W_\lambda=W_S=W$.  We first prove Conjecture \ref{conj} when $\lambda=(n)$.

For the rest of this section we take $\lambda=(n)$. Then
$W_\lambda=N_W(W_\lambda)= W$ and so $E_{(n)} = e_{I_{(n)}} \BBC W_{(n)}$
and $A_{(n)}= A_{X_{(n)}}$. Moreover, $A_{X_{(n)}}= A^{n-1}$ is the top,
non-zero graded piece of $A$.  To simplify the notation, we denote
$A_{(n)}$, $E_{(n)}$, and $e_{I_{(n)}}$ by $A_n$, $E_n$, and $e_n$,
respectively.

Define $c_1=1$ in $W$ and for $1< i\leq n$ define $c_i= s_{i-1} \dotsm
s_2s_1$, so $c_i$ acts on the basis $\{v_1, v_2, \dots, v_n\}$ as an
$i$-cycle. Also, set $c=c_n$. Then,
\begin{itemize}
\item $c$ is a cuspidal element in $W$,
\item the set of cuspidal elements in $W$ is precisely the conjugacy class
  of $c$, and
\item $Z_W(c)= \langle c\rangle$ is the cyclic group of order $n$ generated
  by $c$.
\end{itemize}
Set $\zeta= e^{2\pi i/n}$ in $\BBC$ and define $\varphi\colon Z_W(c)\to
\BBC$ by $\varphi(c\inverse )= \zeta$. The elements we have denoted by $c_i$
are denoted by $c_i\inverse$ by Lehrer and Solomon
\cite{lehrersolomon:symmetric}. However, the character $\varphi$ of $Z_W(c)$
is the same as in \cite{lehrersolomon:symmetric}.

\begin{theorem}\label{top}
  With the preceding notation we have that
  \begin{itemize}
  \item[\sl (a)] the character of $W$ on $E_n$ is $\Ind_{Z_W(c)}^W(\varphi)$
    and
  \item[\sl (b)] the character of $W$ on $A_n$ is $\epsilon \,
    \Ind_{Z_W(c)}^W(\varphi)$.
  \end{itemize}
\end{theorem}

Statement (b) has been proved by Stanley \cite[Theorem 7.2]{stanley:aspects}
and by Lehrer and Solomon \cite[Theorem 3.9]{lehrersolomon:symmetric}. As
mentioned in the introduction, a proof of (a) may be extracted from
classical results about the representation of $S_n$ on the free Lie algebra
on $n$ letters. In contrast, our proof below that the character of $W$ on
$E_n$ is $\Ind_{Z_W(c)}^W(\varphi)$ follows the Lehrer-Solomon argument and
demonstrates a parallelism between the group algebra and the Orlik-Solomon
algebra that we expect will apply in some form to all finite Coxeter
groups. In addition, our argument is valid not only for $E_n$ and $e_n$, but
more generally for $E_{(n)}^\sigma$ and $e_{(n)}^\sigma$ for any function
$\sigma\colon 2^S\to \BBR_{>0}$.

To emphasize and differentiate the parallel arguments, we use the convention
that the superscript $+$ denotes quantities associated with $E_n$ and the
superscript $-$ denotes quantities associated with $A_n$. Notice that with
the notation of \S\ref{sec:expl}, we have $X_c=\{0\}$, and so $\alpha_c=
\det|_{X_c}$ is the trivial character.

Suppose $t$ is an indeterminate. For $0\leq k\leq n$, define elements
$b^+(n,k)$ and $b^-(n,k)$ in $\BBC W$ by
\[
(1 - c_1t)(1 - c_2t) \dotsm (1 - c_{n}t) = \sum_{k=0}^{n} b^+(n,k) t^k
\]
and
\[
(1 + c_{n}t)(1 - c_{n-1}t) \dotsm (1 + (-1)^{n-1}c_{1}t) = \sum_{k=0}^{n}
b^-(n,k) t^k
\]
respectively (the $k\th$ factor in the product on the left-hand side of the
last equation is $(1+ (-1)^{k-1} c_{n-k+1} t)$).

Set $W_{n-1}=\langle s_1, s_2, \dots, s_{n-2} \rangle$. Then $W_{n-1} \cong
S_{n-1}$. The analog of the idempotent $e_n$ in $E_n$ is the basis element
$a_n= a_{s_1} a_{s_2} \dotsm a_{s_{n-1}}$ in $A_n=A^{n-1}$.  Lehrer and
Solomon \cite[\S3]{lehrersolomon:symmetric} prove the following statements:
\begin{itemize}
\item[(i)] $A_n=\BBC W a_n$.
\item[(ii)] $c^{-k} a_n= b^-(n-1,k) a_n$ for $0\leq k\leq n-1$. In
  particular, $A_n= \BBC W_{n-1} a_n$.
\item[(iii)] Consider the homomorphism of left $\BBC W$-modules from $\BBC
  W$ to $A_n$ given by right multiplication by $a_n$. The kernel of this
  mapping is the left $\BBC W_{n-1}$-module generated by $\{\, c^{-k}-
  b^-(n-1,k) \mid 1\leq k\leq n-1\,\}$.
\item[(iv)] $\{\, wa_n\mid w\in W_{n-1}\,\}$ is a $\BBC$-basis of $A_n$ and
  $A_n$ is the left regular $\BBC W_{n-1}$-module.
\end{itemize}
Next we show that the analogous statements hold with $A_n$ replaced by $E_n$
and $b^-(n,k)$ replaced by $b^+(n,k)$.

For $k=1, 2, \dots, n-1$, define $x_k=x_{S\setminus \{s_k\}}$ and
$w_k=c_1c_2\dotsm c_{k}$. Then $w_k$ is the longest element in $\langle s_1,
s_2, \dots, s_{k-1} \rangle$.

\begin{lemma}\label{lem:xk}
  Suppose $1\leq k\leq n-1$. Then
  \[ 
  W^{S\setminus \{s_k\}} w_k=\{\, c_{i_1} \dotsm c_{i_k}\mid 1\leq i_1<\dots
  < i_k\leq n\,\}.
  \]
\end{lemma}

\begin{proof}
  It suffices to show that if $1\leq i_1<\dots <i_k\leq n$, then $c_{i_1}
  \dotsm c_{i_k} w_k$ is in $W^{S\setminus \{s_k\}}$. For this, we consider
  elements in $W$ as acting on $\{1, \dots, n\}$. That is, we identify the
  vector $v_j$ with $j$ for $1\leq j\leq n$. Then
  \[
  W^{S\setminus \{s_k\}}= \{\, w\in W\mid w(j)<w(j+1)\,\forall\, j\in \{ 1,
  \dots, n-1\} \setminus \{k\}\,\}
  \]
  and 
  \[
  w_k(j)= \begin{cases} k+1-j & 1\leq j\leq k\\ j & k+1\leq j\leq n
    .\end{cases}
  \]

  Fix $i_1$, \dots, $i_k$ with $1\leq i_1<\dots <i_k\leq n$ and set $w=
  c_{i_1} \dotsm c_{i_k} w_k$. If $1\leq j\leq k$, then $w(j)= i_j< i_{j+1}=
  w(j+1)$. If $j\geq i_k$, then $w(j)\leq j<j+1=w(j+1)$.  Suppose that
  $k<j<i_k$. Choose $r$ minimal such that
  \[
  j+1\leq i_k,\ j+1-1\leq i_{k-1},\ \dots,\ j+1-r\leq i_{k-r}\text {, and
  }j+1-r-1 > i_{k-r-1}.
  \]
  Then $w(j)\leq j-r-1< j-r= j+1-r-1=w(j+1)$.
\end{proof}

\begin{corollary}\label{cor:xkwk}
  For $1 \leq k \leq n-1$, we have $b^+(n ,k) = (-1)^k x_k w_k$.
\end{corollary}

\begin{proof}
  Using the definition and Lemma \ref{lem:xk} we have
  \[
  b^+(n,k)= (-1)^k \sum_{1\leq i_1<\dots <i_k\leq n} c_{i_1}\dotsm c_{i_k} =
  (-1)^k x_kw_k.
  \]
\end{proof}

\begin{proposition}\label{prop}
  The following analogs of (i)--(iv) above hold.
  \begin{itemize}
  \item[\sl (a)] $E_n=e_n \BBC W$.
  \item[\sl (b)] $e_n c^k = e_n b^+(n-1, k)$ for $0\leq k\leq n-1$. In
    particular, $E_n= e_n \BBC W_{n-1}$.
  \item[\sl (c)] Consider the endomorphism of $\BBC W$ considered as a right
    $\BBC W$-module given by left multiplication by $e_n$. The kernel of
    this mapping is the free, right $\BBC W_{n-1}$-module with basis $\{\,
    c^{k}- b^+(n-1,k) \mid 1\leq k\leq n-1\,\}$.
  \item[\sl (d)] $\{\, e_nw\mid w\in W_{n-1}\,\}$ is a $\BBC$-basis of $E_n$
    and $E_n$ is the right regular $\BBC W_{n-1}$-module.
  \end{itemize}
\end{proposition}

\begin{proof}
  The first statement follows immediately from the definitions.

  We prove (b) by recursion. It is clear that $e_n c^k = e_n b^+(n-1, k)$
  for $k=0$, since $b^+(n-1,0) = 1 = c^0$.  Suppose $e_n c^{k-1} = e_n
  b^+(n-1, k-1)$. It follows from \cite[Theorem 7.8]
  {bergeronbergeronhowletttaylor:decomposition} that $e_nx_J=0$ unless
  $J=S$. Thus, it follows from Corollary \ref{cor:xkwk} that $e_n b^+(n,k)=
  (-1)^k e_n x_k w_k=0$ for $1\leq k\leq n-1$. On the other hand, it follows
  from the definition that
  \[ 
  \textstyle \sum\limits _{k=0}^{n} b^+(n,k) t^k= \left( \sum\limits
    _{k=0}^{n-1} b^+(n-1,k) t^k \right) (1 - c_{n}t)
  \]
  and hence $b^+(n,k) = b^+(n-1,k) - b^+(n-1,k-1) c$ for $1\leq k\leq
  n-1$. Therefore,
  \[
  e_nc^k= e_nc^{k-1} c =e_n b^+(n-1,k-1) c= e_nb^+(n-1,k).
  \]

  Next, consider the endomorphism of $\BBC W$ given by $x\mapsto e_nx$.  Let
  $K$ denote the kernel of this mapping and let $K_1$ denote the $\BBC
  W_{n-1}$-submodule of $\BBC W$ generated by $\{\, c^{k}- b^+(n-1,k) \mid
  1\leq k\leq n-1\,\}$.  It follows from (b) that $K_1\subseteq
  K$. Moreover, $\{\, c^{k}- b^+(n-1,k) \mid 1\leq k\leq n-1\,\}$ is a $\BBC
  W_{n-1}$ basis of $K_1$ because the cyclic subgroup generated by $c$ is a
  left transversal of $W_{n-1}$ in $W$.  Therefore, $\dim_\BBC K_1=
  (n-1)(n-1)!$. However,
  \[
  \dim K= \dim \BBC W- \dim E_n= n!-|W:Z_W(c)|= (n-1) (n-1)!= \dim K_1.
  \] 
  Therefore, $K_1=K$. This proves (c).

  Because $b^+(n-1,k)$ is in $\BBC W_{n-1}$ for $1\leq k\leq n-1$, it
  follows from (b) that the image of the mapping $x\mapsto e_nx$ is $e_n
  \BBC W_{n-1}$. Therefore, $E_n= e_n \BBC W_{n-1}$. Since $\dim E_n=
  (n-1)!$, it follows that $\{\, e_nw\mid w\in W_{n-1}\,\}$ is a
  $\BBC$-basis of $E_n$. This proves (d).
\end{proof}

Finally, define idempotents $f^+$ and $f^-$ in $\BBC Z_W(c)$ by
\[
f^+= {\textstyle \frac 1n} \sum_{k=0}^{n-1} \varphi(c^k) c^{-k} =
{\textstyle \frac 1n} \sum_{k=0}^{n-1} \zeta^{-k} c^{-k}
\]
and
\[
f^-= {\textstyle \frac 1n} \sum_{k=0}^{n-1} \epsilon(c)^k \varphi(c^k)
c^{-k} = {\textstyle \frac 1n} \sum_{k=0}^{n-1} \epsilon(c)^k \zeta^{-k}
c^{-k}.
\]
Obviously, the lines $\BBC f^+$ and $\BBC f^-$ in $\BBC W$ are stable under
left and right multiplication by $Z_W(c)$ and afford the characters
$\varphi$ and $\epsilon \varphi$ of $Z_W(c)$, respectively. Moreover,
$\Ind_{Z_W(c)}^W(\varphi)$ is afforded by the right $\BBC W$-module $f^+\BBC
W$, and $\epsilon \Ind_{Z_W(c)}^W(\varphi) =\Ind_{Z_W(c)}^W(\epsilon
\varphi)$ is afforded by the left $\BBC W$-module $\BBC Wf^-$. Thus, to
prove Theorem \ref{top} it is enough to find $\BBC W$-isomorphisms $E_n
\cong f^+ \BBC W$ and $A_n\cong \BBC W f^-$.

\begin{lemma}\label{invert}
  The idempotent $f^+$ acts invertibly by right multiplication on $e_n$, and
  the idempotent $f^-$ acts invertibly by left multiplication on $a_n$.
\end{lemma}

\begin{proof}
  Lehrer and Solomon \cite[\S3]{lehrersolomon:symmetric} show that $f^-$
  acts invertibly on $a_n$. Their argument is easily modified to show that
  $f^+$ acts invertibly by right multiplication on $e_n$ as follows.

  We have $(1 - c_1\zeta) \dotsm (1 - c_{n-1}\zeta) = \sum_{k=0}^{n-1}
  b^+(n-1,k) \zeta^k$. Multiply both sides on the left by $\frac 1n e_n$ and
  use Proposition \ref{prop}(b) to get
  \[
  \textstyle\frac 1n e_n (1 - \zeta c_1) \dotsm (1 -\zeta c_{n-1}) = \frac
  1n \sum\limits _{k=0}^{n-1}\zeta^k e_nb^+(n-1,k) = \frac 1n \sum\limits
  _{k=0}^{n-1} \zeta^k e_n c^k= e_n f^+.
  \]
  If $1\leq k\leq n-1$, then
  \[
  1-\zeta^k =1-\zeta^kc_k^k= (1-\zeta c_k)(1+\zeta c_k+ \dots + \zeta^{k-1}
  c_k^{k-1}).
  \]
  Since $\zeta$ is a primitive $n\th$ root of unity, $1-\zeta^k\ne 0$ in
  $\BBC$. Thus, $1-\zeta c_k$ acts invertibly on $e_n$ for $1\leq k\leq n-1$
  and so $f^+$ acts invertibly on $e_n$.
\end{proof}

\begin{proof}[Proof of Theorem \ref{top}]
  (See \cite[\S3]{lehrersolomon:symmetric}.) Consider the mapping from
  $f^+\BBC W$ to $E_n$ given by $x\mapsto e_nx$. It follows from Lemma
  \ref{invert} and the discussion preceding it that $e_nf^+\ne0$, that
  $Z_W(c)$ acts on the line $\BBC e_n f^+$ in $E_n$ as the character
  $\varphi$, and that the mapping is a surjection. Since $\dim f^+\BBC W=
  |W:Z_W(c)|= (n-1)!= \dim E_n$, the mapping is also an injection. Thus, we
  have an isomorphism of right $\BBC W$-modules, $E_n\cong f^+\BBC W$.

  As in \cite[\S3]{lehrersolomon:symmetric}, similar reasoning applies to
  the mapping from $\BBC Wf^-$ to $A_n$ given by $x\mapsto xa_n$ and shows
  that $A_n\cong \BBC W f^-$.
\end{proof}

\section{Symmetric groups: arbitrary $\lambda$ }\label{arbitrary}

In this section we consider the case of an arbitrary partition of $n$ and
complete the proof of Conjecture \ref{conj} for symmetric groups.

Suppose $\lambda=(\lambda_1, \lambda_2, \dots, \lambda_p)$ is a partition of
$n$. Recall from \S\ref{lambda=n} that $I_\lambda= S\setminus \{s_{\tau_1},
s_{\tau_2}, \dots, s_{\tau_{p-1}} \}$ and that $W_\lambda=\langle I_\lambda
\rangle$ is isomorphic to the product of symmetric groups
$S_{\lambda_1}\times \dots \times S_{\lambda_p}$, where the factor
$S_{\lambda_i}$ acts on $\{v_{\tau_{i-1}+1}, v_{\tau_{i-1}+2}, \dots,
v_{\tau_{i}} \}$. For $1\leq i\leq p$ define $g_{\lambda_i}= s_{\tau_{i}-1}
\dotsm s_{\tau_{i-1}+2} s_{\tau_{i-1}+1}$. Then $g_{\lambda_i}$ is the
$\lambda_i$-cycle in $S_{\lambda_i}$ that corresponds to the $n$-cycle
$c=c_n$ in \S\ref{lambda=n}. Next, define $c_\lambda=g_{\lambda_1}
g_{\lambda_2} \dotsm g_{\lambda_p}$ and $Z_\lambda=Z_{W_\lambda}
(c_\lambda)$. Then
\begin{itemize}
\item $c_\lambda$ is a cuspidal element in $W_\lambda$,
\item the set of cuspidal elements in $W_\lambda$ is precisely the conjugacy
  class of $c_\lambda$, and
\item $Z_\lambda \cong \langle g_{\lambda_1} \rangle \times \langle
  g_{\lambda_2} \rangle \times \dots \times \langle g_{\lambda_p} \rangle$.
\end{itemize}

Notice that $\{\, c_\lambda\mid \text{$\lambda$ is a partition of $n$}\,\}$
is a complete set of conjugacy class representatives in $W$.

With $\lambda$ as above, for $1\leq i\leq p$, define $\varphi_{\lambda_i}$
to be the character of $\langle g_{\lambda_i} \rangle$ with
$\varphi_{\lambda_i}( g_{\lambda_i} \inverse)= e^{2\pi i/\lambda_i}$. Then
$\varphi_{\lambda_i}$ is the analog of the character $\varphi$ in
\S\ref{lambda=n} for the factor $S_{\lambda_i}$ of $W_\lambda$. Next, define
the character $\varphi_\lambda$ of $Z_\lambda \cong \langle g_{\lambda_1}
\rangle \times \langle g_{\lambda_2} \rangle \times \dots \times \langle
g_{\lambda_p} \rangle$ to be
\[
\varphi_\lambda=\varphi_{\lambda_1} \otimes \dotsm \otimes
\varphi_{\lambda_p}.
\]
Note that this notation is not consistent with that of Lehrer and Solomon;
our character $\varphi_\lambda$ corresponds to the character
$\varphi_\lambda \epsilon$ in \cite{lehrersolomon:symmetric}. Applying the
special case $\lambda=(n)$ considered in \S\ref{lambda=n} to each factor
$S_{\lambda_i}$ of $W_\lambda$, for $1\leq i\leq p$ define
\[
f_{\lambda_i}^+= {\textstyle \frac 1{\lambda_i}} \sum_{k=0}^{\lambda_i-1}
\varphi_{\lambda_i}(g_{\lambda_i}^k) g_{\lambda_i}^{-k} \quad \text{and}
\quad f_{\lambda_i}^-= {\textstyle \frac 1{\lambda_i}}
\sum_{k=0}^{\lambda_i-1} \epsilon(g_{\lambda_i}^k) \varphi(g_{\lambda_i}^k)
g_{\lambda_i}^{-k}.
\]
Finally, define idempotents $f_\lambda^+$ and $f_\lambda^-$ in $\BBC
Z_\lambda$ by
\[
f_\lambda^+= f_{\lambda_1}^+ f_{\lambda_2}^+ \dotsm f_{\lambda_p}^+ \quad
\text{and} \quad f_\lambda^-= f_{\lambda_1}^- f_{\lambda_2}^- \dotsm
f_{\lambda_p}^-.
\]
Obviously, the lines $\BBC f_\lambda^+$ and $\BBC f_\lambda^-$ in $\BBC W$
are stable under left and right multiplication by $Z_\lambda$ and afford the
characters $\varphi_\lambda$ and $\epsilon \varphi_\lambda$ of $Z_\lambda$,
respectively.

Now consider the canonical complement $N_{X_\lambda}$ of $W_\lambda$ in
$N_W( W_\lambda)$. Set $N_\lambda=N_{X_\lambda}$. If $\lambda$ has $m_j$
parts equal to $j$, then $N_\lambda$ is isomorphic to the product of
symmetric groups $\prod_j S_{m_j}$ (see \cite{howlett:normalizers} or
\cite{lehrersolomon:symmetric}). In particular, $N_\lambda$ has one Coxeter
generator, say $r_i$, for each $i$ such that $\lambda_i=\lambda_{i+1}$. The
generator $r_i$ acts on the set $\{v_1, v_2, \dots, v_n\}$ by interchanging
$v_{\tau_{i-1}+j}$ and $v_{\tau_{i}+j}$ for $1\leq j\leq \lambda_i$, and
fixing $v_k$ for $k\leq \tau_{i-1}$ and $k> \tau_{i+1}$.

It is well-known and easy to check (\cite{lehrersolomon:symmetric},
\cite{konvalinkapfeifferroever:centralizers}) that $N_\lambda\subseteq
Z_W(c_\lambda)$, and so $Z_W(c_\lambda) \cong Z_\lambda \rtimes N_\lambda$.

\begin{lemma}\label{lem:stab}
  The subgroup $N_\lambda$ of $Z_W(c_\lambda)$ stabilizes the characters
  $\varphi_\lambda$ and $\epsilon \varphi_\lambda$ of $Z_\lambda$ and
  centralizes the idempotents $f_\lambda^+$ and $f_\lambda^-$. In
  particular, $\varphi_\lambda$ extends to a character, also denoted by
  $\varphi_\lambda$, of $Z_W(c_\lambda)$, with
  $\varphi_\lambda(nz)=\varphi_\lambda(z)$ for $n$ in $N_\lambda$ and $z$ in
  $Z_\lambda$.
\end{lemma}

\begin{proof}
  Suppose that $i$ is such that $\lambda_i=\lambda_{i+1}$ and consider the
  generator $r_i$ of $N_\lambda$. Then $r_i$ is an involution and it follows
  from the description of the action of $r_i$ on the basis $\{ v_1, \dots,
  v_n\}$ of $V$ that
  \[
  r_i g_{\lambda_j} r_i\inverse= r_i g_{\lambda_j} r_i=
  \begin{cases}  g_{\lambda_{i+1}}& j=i\\ g_{\lambda_{i}} &j=i+1\\
    g_{\lambda_j}&j\ne i, i+1.
  \end{cases}
  \]
  Since $\varphi_\lambda(g_{\lambda_i})= \varphi_\lambda(
  g_{\lambda_{i+1}})$, it follows that $r_i$ stabilizes $\varphi_\lambda$
  and $\epsilon \varphi_\lambda$.

  The group $N_\lambda$ is generated by $\{\, r_i\mid \lambda_i=
  \lambda_{i+1} ,\}$ and so $N_\lambda$ stabilizes the characters
  $\varphi_\lambda$ and $\epsilon \varphi_\lambda$ of $Z_\lambda$. Moreover,
  $N_\lambda$ acts on $\{ g_{\lambda_1}, \dots, g_{\lambda_p}\}$ by
  conjugation as a group of permutations. Thus, it follows from the
  definition of $f_{\lambda_i}^+$ and $f_{\lambda_i}^-$ that conjugation by
  $N_\lambda$ permutes $\{ f_{\lambda_1}^+, \dots, f_{\lambda_p}^+\}$ and
  $\{ f_{\lambda_1}^-, \dots, f_{\lambda_p}^-\}$. Since the
  $f_{\lambda_i}^+$'s pairwise commute and the $f_{\lambda_i}^-$'s pairwise
  commute, we see that $N_\lambda$ centralizes both $f_{\lambda_1}^+ \dotsm
  f_{\lambda_p}^+ =f_\lambda^+$ and $f_{\lambda_1}^- \dotsm f_{\lambda_p}^-
  =f_\lambda^-$.
\end{proof}

Set $\alpha_\lambda= \alpha_{X_\lambda}$. Then $\alpha_\lambda$ is a
character of $N_W(W_\lambda)$ and $\alpha_\lambda (r_i)= -1$.  Note that
this notation is not consistent with that of Lehrer and Solomon; our
character $\alpha_\lambda$ corresponds to the character $\alpha_\lambda
\epsilon$ in \cite{lehrersolomon:symmetric} as
$\epsilon(r_i)=(-1)^{\lambda_i}$.

\begin{theorem}\label{any}
  Suppose that $\lambda$ is a partition of $n$. Then the
  $N_W(W_\lambda)$-modules $e_{I_\lambda} \BBC W_\lambda$ and
  $A_{X_\lambda}$, and the character $\varphi_\lambda$ of $Z_W(c_\lambda)$,
  are related by
  \begin{itemize}
  \item[(a)] the character of the right $N_W(W_\lambda)$-module
    $e_{I_\lambda} \BBC W_\lambda$ is $\Ind_{Z_W(c_\lambda)} ^{
      N_W(W_\lambda)} (\varphi_\lambda)$ and
  \item[(b)] the character of the left $N_W(W_\lambda)$-module
    $A_{X_\lambda}$ is $\epsilon\, \alpha_\lambda\, \Ind_{Z_W (c_\lambda)}
    ^{ N_W(W_\lambda)} (\varphi_\lambda)$.
  \end{itemize} 
\end{theorem}

\begin{proof}
  Statement (b) has been proved by Lehrer and Solomon \cite[Theorem
  4.4]{lehrersolomon:symmetric}. Their argument may be rephrased as
  follows. Extending the definition of the element $a_n$ in $A_n$ when
  $\lambda=(n)$, Lehrer and Solomon define an element $a_\lambda$ in
  $A_{X_\lambda}$ on which $f_\lambda^-$ acts invertibly. Then:
  \begin{itemize}
  \item[(i)] $Z_W(c_\lambda)$ acts on the line $\BBC f_\lambda^- a_\lambda$
    in $A_{X_\lambda}$ via the character $\epsilon_\lambda \alpha_\lambda
    \varphi_\lambda$.
  \item[(ii)] $A_{X_\lambda}= \BBC N_W(W_\lambda) f_\lambda^- a_\lambda$.
  \item[(iii)] The multiplication map $\BBC N_W(W_\lambda) \otimes_{\BBC
      Z_W(c_\lambda)} \BBC f_\lambda^-a_\lambda \to A_{X_\lambda}$ is an
    isomorphism.
  \end{itemize}
  Therefore, $A_{X_\lambda} \cong \Ind_{Z_W (c_\lambda)} ^{ N_W(W_\lambda)}
  (\BBC f_\lambda^- a_\lambda )$ and hence the character of $A_{X_\lambda}$
  is indeed $\Ind_{Z_W (c_\lambda)} ^{ N_W(W_\lambda)} ( \epsilon_\lambda
  \alpha_\lambda \varphi_\lambda)$.

  To prove (a) we first note that by Proposition \ref{pro:commutes},
  $e_{I_\lambda} \BBC W_\lambda$ and $e_{I_\lambda} ^{I_\lambda} \BBC
  W_\lambda$ are isomorphic right $N_W(W_\lambda)$-modules and so it
  suffices to prove that $e_{I_\lambda} ^{I_\lambda} \BBC W_\lambda$ affords
  the character $\Ind_{Z_W(c_\lambda)} ^{ N_W(W_\lambda)}
  (\varphi_\lambda)$. We argue as for $A_{X_\lambda}$ with $e_{I_\lambda}
  ^{I_\lambda}$ in place of $a_\lambda$.

  For the rest of this proof we fix a partition $\lambda=(\lambda_1, \dots,
  \lambda_p)$ of $n$. To simplify the notation, set $I=I_\lambda$ and
  $e=e_I^I$.  It suffices to show that the line $\BBC e f_\lambda^+$ in the
  right $N_W(W_\lambda)$-module $e \BBC W_\lambda$ satisfies properties
  analogous to (i), (ii), and (iii) above.

  \begin{itemize}
  \item[(i$'$)] {\sl $Z_W(c_\lambda)$ acts on the line $\BBC e f_\lambda^+$
      via the character $\varphi_\lambda$:} We have seen in Lemma
    \ref{lem:n} that $N_\lambda$ centralizes $e$ and in Lemma \ref{lem:stab}
    that $N_\lambda$ centralizes $f_\lambda^+$.  Thus, $N_\lambda$
    centralizes $e f_\lambda^+$ and $Z_W( c_\lambda) =N_\lambda Z_\lambda$
    acts on the line $\BBC e f_\lambda^+$ via the character
    $\varphi_\lambda$ if $e f_\lambda^+\ne 0$.

    Let $\tau \colon 2^I\to \BBR_{>0}$ be the function that takes the
    constant value $1$. We have $I=\coprod_{j=1}^p I_j$ where
    $W_{\lambda_j}= \langle I_j \rangle$ and so $e=e_I^\tau = e_{I_1} ^\tau
    \dotsm e_{I_p} ^\tau$ by Proposition \ref{pro:red}. Therefore,
    \[
    e f_\lambda^+= (e_{I_1}^{\tau} \dotsm e_{I_p} ^{\tau}) \, (
    f_{\lambda_1}^+ \dotsm f_{\lambda_p}^+) = (e_{I_1}^{\tau} f_{I_1}^+)
    \dotsm (e_{I_p}^{\tau} f_{\tau_p}^+).
    \]
    For $1\leq j\leq p$, the idempotent $e_{I_j}^{\tau}$ in $\BBC
    W_{\lambda_j}$ is defined using the partition $(\lambda_j)$ of
    $\lambda_j$ as in \S\ref{lambda=n} and so the idempotent
    $f_{\lambda_j}^+$ acts as a unit on $e_{I_j}^{\tau}$ by Lemma
    \ref{invert}. Therefore, $f_\lambda^+$ acts invertibly by right
    multiplication on $e$ and so $ef_\lambda^+\ne0$.

  \item[(ii$'$)] {\sl $e\BBC W_\lambda= e f_\lambda^+ \BBC N_W(W_\lambda)$:}
    Because $f_\lambda^+$ acts invertibly on $e$ and $N_W(W_\lambda)=
    N_\lambda Z_\lambda W_\lambda$, we have
    \[
    e \BBC W_\lambda= e f_\lambda^+\BBC W_\lambda = ef_\lambda^+ \BBC
    N_\lambda Z_\lambda W_\lambda = e f_\lambda^+ \BBC N_W(W_\lambda) .
    \]

  \item[(iii$'$)] {\sl The multiplication map $\BBC ef_\lambda^+
      \otimes_{\BBC Z_W(c_\lambda)} \BBC N_W(W_\lambda) \to e\BBC W_\lambda$
      is an isomorphism:} It follows from (ii$'$) that the mapping is
    surjective. Moreover, using Corollary \ref{cor:dimA=dimE} we have
    \begin{align*}
      \dim e f_\lambda^+ \BBC N_W(W_\lambda) &= \dim e \BBC W_\lambda\\
      &= |W_\lambda: Z_\lambda|\\
      &= |N_W(W_\lambda): Z_W(c_\lambda)|\\
      &= \dim \BBC ef_\lambda^+ \otimes_{\BBC Z_W(c_\lambda)} \BBC
      N_W(W_\lambda)
    \end{align*}
    and so the mapping is an isomorphism.
  \end{itemize}

  This completes the proof of the theorem.
\end{proof}

The proof of Conjecture \ref{conj} for symmetric groups now follows from
Proposition \ref{pro:induct}, Theorem \ref{any}, and transitivity of
induction.

\begin{theorem}\label{thm:main}
  For each partition $\lambda$ of $n$ there is a linear character
  $\varphi_\lambda$ of $Z_W(c_\lambda)$ such that
  \begin{itemize}
  \item[\sl (a)] the character of $E_\lambda$ is $\Ind_{Z_W (c_\lambda) }^W
    (\varphi_\lambda)$ and
  \item[\sl (b)] the character of $A_\lambda$ is $\Ind_{Z_W(c_\lambda)}^W
    (\epsilon_\lambda \alpha_\lambda \varphi_\lambda )$, where
    $\epsilon_\lambda$ denotes the restriction of $\epsilon$ to
    $Z_W(c_\lambda)$.
  \end{itemize}
  In particular,
  \[
  H^p(M_W)\cong \bigoplus_{\lambda\vdash n, \,\rk(c_\lambda)=p} \Ind_{
    Z_W(c_\lambda)}^W (\epsilon_\lambda \alpha_\lambda \varphi_\lambda)
  \]
  for $0\leq p\leq n-1$, and
  \[
  \BBC W\cong \bigoplus_{\lambda\vdash n} \Ind_{Z_W
    (c_\lambda)}^W(\varphi_\lambda) \quad\text{and}\quad A\cong \bigoplus_{
    \lambda\vdash n} \Ind_{Z_W(c_\lambda)}^W(\epsilon_\lambda \alpha_\lambda
  \varphi_\lambda).
  \]
\end{theorem}

\section{Parabolic subgroups of type $A$}\label{sec:rel}

In this section we return to the case when $W$ is an arbitrary finite
Coxeter group and prove a relative version of Theorem~\ref{thm:main}.

Suppose that $\lambda$ is in $\Lambda$, $c$ is in $W$ with $\sh(c)=\lambda$,
and that all the irreducible components of $W_c$ are of type $A$. Without
loss of generality we may assume that $W_c=W_L$ is a standard parabolic
subgroup. Suppose that $L=\coprod_{i=1}^p L_i$ where $|L_i|=l_i$ and each
$L_i$ is of type $A_{l_i}$. Define $n_i=l_i+1$. Then
\[
W_L \cong W_{L_1} \times \dots \times W_{L_p} \cong S_{n_1} \times \dots
\times S_{n_p}.
\]
We assume that $l_1\geq \dots \geq l_p$, $L_i = \{s_{i,1}, \dots, s_{i,
  l_i}\}$, and $\Delta_{L_i}= \{\alpha_{i,1}, \dots, \alpha_{i,l_i} \,\}$,
where the labeling is such that $s_{i,j}$ and $s_{i,k}$ commute if
$|j-k|>1$.

Define $c_i= s_{i,l_i} \cdots s_{i,2} s_{i,1}$ in $W_{L_i}$ and $c=c_1\dotsm
c_p$. Then $c$ is a cuspidal element in $W_L$ and the set of cuspidal
elements in $W_L$ is precisely the $W_L$-conjugacy class of $c$. For $1\leq
i\leq p$ define $\varphi_{i}$ to be the character of $\langle c_i \rangle$
with $\varphi_i( c_{i} \inverse)= \zeta_{n_i}$ where $\zeta_{n_i}$ is a
fixed primitive $n_i\th$ root of unity, and set $\varphi_c= \varphi_1\otimes
\dotsm \otimes \varphi_p$. Then $\varphi_c$ is a character of
$Z_{W_L}(c)\cong \langle c_1\rangle \times \dotsm \times \langle c_p
\rangle$.

The rest of this section is devoted to the proof of the following theorem.

\begin{theorem}\label{thm:rel}
  The character $\varphi_c$ of $Z_{W_L}(c)$ extends to a character
  $\widetilde{\varphi}_c$ of $Z_W(c)$ such that
  \[
  E_\lambda \cong \Ind_{Z_W(c)}^W(\widetilde{\varphi}_c) \quad\text{and}
  \quad A_{\lambda} \cong \Ind_{Z_W(c)}^W(\epsilon_c \alpha_c
  \widetilde{\varphi}_c ).
  \]
\end{theorem}

The proof follows the same outline as in \S\ref{arbitrary}: We find lines in
$e_{L} ^{L} \BBC W_L$ and $A_{X_L}$ such that the analogs of statements (i),
(ii), and (iii) and (i$'$), (ii$'$), and (iii$'$) in the proof of
Theorem~\ref{any} hold, and so
\[
e_{L} ^{L} \BBC W_L \cong \Ind_{Z_W(c)} ^{ N_W(W_L)} (\widetilde{\varphi}_c)
\quad \text{and}\quad A_{X_L} \cong \Ind_{Z_W (c)} ^{ N_W(W_L)} (\epsilon
\alpha_c \widetilde{\varphi}_c).
\]
We then apply $\Ind_{N_W(W_L)}^W$ to both sides of both equations and the
theorem follows from Proposition~\ref{pro:induct} by transitivity of
induction. The argument in this section is complicated by the fact that the
subgroup $N_L$ is not necessarily contained in $Z_W(c)$.

In case $W$ is a symmetric group, it was shown in~\ref{lem:stab} that
$\widetilde{\varphi}_c$ is the trivial extension of $\varphi_c$. In the
general case, this is no longer so. Formulas for $\widetilde{\varphi}_c$ are
given in the proof of Lemma~\ref{lem:i} below. Notice that $c$ is an
involution if and only if $l_1=1$ for $1\leq i\leq p$, and that in this case
$\varphi_c$ is the sign character of $Z_{W_L}(c)$ and
$\widetilde{\varphi}_c$ is the trivial extension of $\varphi_c$ to $Z_W(c)$.

Although $N_L$ is not necessarily contained in $Z_W(c)$, Konvalinka,
Pfeiffer, and R\"over \cite{konvalinkapfeifferroever:centralizers} have
shown that $Z_{W_L}(c)$ does have a complement, $N_c$, in $Z_W(c)$, and
$N_c$ is also a complement to $W_L$ in $N_W(W_L)$.

By \cite{howlett:normalizers}, the group $N_L$ is generated by $\{\, r_i,
g_j\mid 1\leq i\leq p-1,\, 1\leq j\leq p\,\}$, where $r_i$ and $g_i$ act on
$W_L$ as follows:
\begin{itemize}
\item If $n_i\ne n_{i+1}$, then $r_i=1$. If $n_i=n_{i+1}$, then $r_i$ acts
  on $L$ by
  \[
  r_i s_{j,k} r_i\inverse =
  \begin{cases}
    s_{i+1,k} &\text{if $j=i$}\\ s_{i,k} &\text{if $j=i+1$} \\ s_{j,k}
    &\text{otherwise.}
  \end{cases}
  \]
  In particular, $r_i$ is an involution,
  \[
  r_ic_jr_i=
  \begin{cases}
    c_{i+1} &\text{if $j=i$}\\ c_{i} &\text{if $j=i+1$} \\ c_{j}
    &\text{otherwise,}
  \end{cases}
  \]
  and $r_i$ is in $Z_W(c)$.
\item Either $g_i=1$ or $g_i$ acts on $L$ by
  \[
  g_i s_{j,k} g_i\inverse =
  \begin{cases}
    s_{i,n_i-k} &\text{if $j=i$} \\ s_{j,k} &\text{if $j\ne i$.}
  \end{cases}
  \]
  In particular, $g_i$ is an involution and if $g_i\ne1$, then
  \[
  g_ic_jg_i=
  \begin{cases}
    c_{i}\inverse &\text{if $j=i$}\\ c_{j} &\text{if $j\ne i$.}
  \end{cases}
  \]
\end{itemize}
Notice that if $W$ is a symmetric group, then $g_i=1$ for $1\leq i\leq p$.

For $1\leq i\leq p$, let $w_i$ denote the longest element in $W_{L_i}$ and
define
\[
h_i=
\begin{cases}
  1&\text{if $g_i=1$}\\ g_iw_i &\text{if $g_i\ne 1$.}
\end{cases}
\]
Then $h_i$ is in $Z_W(c_j)$ for $1\leq i,j\leq p$. It is shown in
\cite{konvalinkapfeifferroever:centralizers} that
\[
N_c = \langle r_i, h_j \mid 1\leq i\leq p-1,\, 1\leq j\leq p \rangle
\]
is a complement to $Z_{W_L}(c)$ in $Z_W(c)$ and to $W_L$ in $N_W(W_L)$.

As in \S\ref{arbitrary} define
\begin{align*}
  f_{i}^+&= {\textstyle \frac 1{n_i}} \sum_{k=0}^{n_i-1}
  \varphi_{i}(c_{i}^k) c_{i}^{-k} & f_{i}^-&= {\textstyle \frac 1{n_i}}
  \sum_{k=0}^{n_i-1} \epsilon(c_{i}^k) \varphi_i(c_{i}^k) c_{i}^{-k} \\
  f_L^+&= f_{_1}^+ f_{_2}^+ \dotsm f_{p}^+ & f_L^-&= f_{1}^- f_{2}^- \dotsm
  f_{p}^-.
\end{align*}
Then the lines $\BBC f_L^+$ and $\BBC f_L^-$ in $\BBC W$ are stable under
left and right multiplication by $Z_{W_L}(c)$ and afford the characters
$\varphi_c$ and $\epsilon \varphi_c$ of $Z_{W_L}(c)$, respectively. Because
$h_i$ centralizes $c_j$ for $1\leq i, j\leq p$, the proof of the second
statement in Lemma~\ref{lem:stab} applies word-for-word to $N_c$ and proves
the next lemma.

\begin{lemma} 
  The subgroup $N_c$ of $Z_W(c)$ centralizes the idempotents $f_L^+$ and
  $f_L^-$ in $\BBC Z_{W_L}(c)$.
\end{lemma}

For $1\leq i\leq p$, define $a_i= a_{s_{i,1}} \dotsm a_{s_{i,l_i}}$. Set
$a_L=a_1\dotsm a_p$. Then $a_L$ is in $A_{X_L}$. The next lemma is the
analog of statements (i) and (i$'$) in the proof of Theorem~\ref{any}.

\begin{lemma}\label{lem:i}
  The lines $\BBC e_L^L f_L^+$ in $e_L^L \BBC W_L$ and $\BBC f_L^-a_L$ in
  $A_{X_L}$ are non-zero and $Z_W(c)$-stable. Let $\widetilde{\varphi}_c$ be
  the character of $Z_W(c)$ acting by right multiplication on the line $\BBC
  e_L^L f_L^+$. Then $\widetilde{\varphi}_c$ is an extension of $\varphi_c$,
  and the character of $Z_W(c)$ acting by left multiplication on the line
  $\BBC f_L^- a_L$ is $\epsilon\alpha_c\widetilde{\varphi}_c$.
\end{lemma}

\begin{proof}
  The argument in statement (i$'$) in \S\ref{arbitrary} applies verbatim to
  $f_L^+$ and shows that $f_L^+$ acts invertibly by right multiplication on
  $e_L^L$ and so $e_L^L f_L^+\ne0$. Similarly, the argument in Lehrer and
  Solomon \cite[\S3]{lehrersolomon:symmetric} shows that $f_i^-$ acts as a
  unit on $a_i$ for $1\leq i\leq p$. It follows that $f_L^-$ acts invertibly
  by left multiplication on $a_L$ and so $f_L^-a_L\ne0$. Therefore,
  $Z_{W_L}(c)$ acts on $\BBC e_L^L f_L^+$ and $\BBC f_L^-a_L$ by $\varphi_c$
  and $\epsilon \varphi_c$, respectively, and $\widetilde{\varphi}_c$
  extends $\varphi_c$.

  If $n_i=n_{i+1}$, then as in \S\ref{arbitrary} we have
  \begin{equation}
    \label{eq:x3}
    (e_L^L f_L^+)\cdot r_i= (e_L^L \cdot r_i) f_L^+ = (r_i e_L^L r_i) f_L^+
    = e_L^L f_L^+  
  \end{equation}
  and
  \begin{multline}
    \label{eq:x4}
    r_i\cdot f_L^-a_L= f_L^-(r_i\cdot a_L) = f_L^-\cdot a_1\dotsm a_{i-1} a_{i+1}
    a_i a_{i+2} \dotsm a_p \\ = (-1)^{l_i^2} f_L^-\cdot a_L = (-1)^{l_i}
    f_L^-\cdot a_L. 
  \end{multline}
  If $h_i\ne 1$, then
  \begin{equation}
    \label{eq:x5}
    e_L^L f_L^+ \cdot h_i= (e_L^L \cdot g_iw_i) f_L^+ = (g_i e_L^L g_iw_i)
    f_L^+ = (e_L^L w_i) f_L^+ = (-1)^{l_i} e_L^L f_L^+,     
  \end{equation}
  (the last equality follows from Lemma~\ref{lem:w0}), and
  \begin{equation}
    \label{eq:x6}
    h_i\cdot f_L^-a_L= f_L^-(h_i\cdot a_L) = f_L^- a_L,
  \end{equation}
  since $h_i$ centralizes $L$. It follows that $Z_W(c)$ acts on the lines
  $\BBC e_L^L f_L^+$ and $\BBC f_L^-a_L$. Moreover, from (\ref{eq:x3}) and
  (\ref{eq:x5}) we see that
  \[
  \widetilde{\varphi}_c(r_i) =1\quad\text{and}\quad
  \widetilde{\varphi}_c(h_i) =(-1)^{l_i}.
  \]

  To complete the proof we need to show that $Z_W(c)$ acts on the line $\BBC
  f_L^- a_L$ as $\epsilon\alpha_c\widetilde{\varphi}_c$.

  For $w$ in $Z_{W_L}(c)$ we have $\alpha_c(w)=1$ and
  $\widetilde{\varphi}_c(w) = \varphi_c(w)$.  Hence
  \[
  w\cdot f_L^-a_L= \epsilon(w) \varphi_c(w) \cdot f_L^- a_L= \bigl(\epsilon
  \alpha_c \widetilde{\varphi}_c\bigr)(w) \, f_L^- a_L.
  \]

  For $n$ in $N_L$ it is shown in
  \cite[Lemma~2.1]{douglasspfeifferroehrle:inductive} that $\epsilon(n)
  \alpha_c(n)$ is the sign of the permutation of $L$ induced by conjugation
  by $n$.

  Suppose next that $1\leq i\leq p-1$ and $n_i=n_{i+1}$. Then the sign of
  the permutation of $L$ induced by conjugation by $r_i$ is $(-1)^{l_i}$ and
  so $\epsilon(r_i) \alpha_c(r_i)= (-1)^{l_i}$.  Therefore,
  using~(\ref{eq:x4}) we see that
  \[
  r_i\cdot f_L^-a_L= (-1)^{l_i} f_L^- a_L = \bigl(\epsilon \alpha_c
  \widetilde{\varphi}_c\bigr)(r_i) \, f_L^- a_L.
  \]

  Finally, suppose that $1\leq i\leq p$ and that $g_i\ne 1$. Then the sign
  of the permutation of $L$ induced by conjugation by $g_i$ is $1$ if
  $l_i\equiv 0,1\pmod 4$ and $-1$ if $l_i\equiv 2,3\pmod 4$. Thus,
  $\epsilon(g_i) \alpha_c(g_i)= 1$ if $l_i\equiv 0,1\pmod 4$ and
  $\epsilon(g_i) \alpha_c(g_i)= -1$ if $l_i\equiv 2,3\pmod 4$.  Since
  $\ell(w_i)=\binom{l_i+1}{2}$, it follows that $\epsilon(w_i)=1$ if
  $l_i\equiv 0,3\pmod 4$ and $\epsilon(w_i)=-1$ if $l_i\equiv 1,2\pmod
  4$. Hence $\epsilon(g_i) \alpha_c(g_i) \epsilon(w_i)= (-1)^{l_i}$. Because
  $\alpha_c(w_i)=1$ we conclude that
  \begin{equation*}
    \label{eq:x2}
    \epsilon(h_i) \alpha_c(h_i)= \epsilon(g_i)\epsilon(w_i) \alpha_c(g_i)
    \alpha_c(w_i) = (-1)^{l_i} = \widetilde{\varphi}_c(h_i).
  \end{equation*}
  Therefore, using ~(\ref{eq:x6}) we see that
  \[
  h_i\cdot f_L^-a_L= f_L^- a_L = \bigl(\epsilon \alpha_c
  \widetilde{\varphi}_c\bigr)(h_i) \, f_L^- a_L.
  \]
  This completes the proof of the lemma.
\end{proof}

Because $f_\lambda^+$ acts invertibly on $e_L^L$, $N_c$ acts on the line
$\BBC e_L^L f_L^+$ by scalars, and $N_W(W_L)=N_cW_L$, we see that
\begin{equation}
  \label{eq:iie}
  e_L^L \BBC W_L= e_L^L f_L^+ \BBC W_L= e_L^L f_L^+ \BBC N_cW_L= e_L^L f_L^+
  \BBC N_W(W_L) .  
\end{equation}
Lehrer and Solomon \cite[Proposition 4.4(ii)]{lehrersolomon:symmetric} have
shown that $A_{X_L}= \BBC W_L a_L$. Thus, because $f_\lambda^-$ acts
invertibly on $a_L$, $N_c$ acts on the line $\BBC f_L^-a_L$ by scalars, and
$N_W(W_L)=W_LN_c$, we see that
\begin{equation}
  \label{eq:iia}
  A_{X_L}= \BBC W_L a_L= \BBC W_L f_L^- a_L= \BBC W_LN_c f_L^- a_L= \BBC
  N_W(W_L) f_L^- a_L.  
\end{equation}
Equations~(\ref{eq:iie}) and~(\ref{eq:iia}) are the analogs of statements
(ii) and (ii$'$) in the proof of Theorem~\ref{any}. The dimension
computation in the proof of statement (iii$'$) now applies to show that the
analogs of statements (iii) and (iii$'$) both hold in the present situation:
The multiplication maps
\begin{equation}
  \label{eq:iii}
  \BBC e_L^Lf_L^+ \otimes_{\BBC Z_W(c)} \BBC N_W(W_L) \to e_L^L\BBC W_L
  \quad\text{and}\quad \BBC N_W(W_L) \otimes_{\BBC Z_W(c)} \BBC f_L^-a_L \to
  A_{X_L}  
\end{equation}
are isomorphisms of $\BBC N_W(W_L)$-modules.

It follows from~(\ref{eq:iii}) and Lemma~\ref{lem:i} that
\[
e_{L} ^{L} \BBC W_L \cong \Ind_{Z_W(c)} ^{ N_W(W_L)} (\widetilde{\varphi}_c)
\quad \text{and}\quad A_{X_L} \cong \Ind_{Z_W (c)} ^{ N_W(W_L)} (\epsilon
\alpha_c \widetilde{\varphi}_c).
\]
Apply $\Ind_{N_W(W_L)}^W$ to both sides of these last two equations and use
transitivity of induction to get
\[
\Ind_{N_W(W_L)}^W \left(e_{L} ^{L} \BBC W_L\right) \cong \Ind_{Z_W(c)} ^{ W}
(\widetilde{\varphi}_c)
\]
and
\[
\Ind_{N_W(W_L)}^W \left(A_{X_L} \right)\cong \Ind_{Z_W (c)} ^{W} (\epsilon
\alpha_c \widetilde{\varphi}_c).
\]
By Proposition~\ref{pro:induct} and Corollary~\ref{cor:conc} we have
\[
E_\lambda \cong \Ind_{Z_W(c)} ^{ W} (\widetilde{\varphi}_c)
\quad\text{and}\quad A_\lambda \cong \Ind_{Z_W (c)} ^{W} (\epsilon \alpha_c
\widetilde{\varphi}_c).
\]
This completes the proof of the theorem.


\bigskip {\bf Acknowledgments:} The authors would like to acknowledge
support from the DFG-priority programme SPP1489 ``Algorithmic and
Experimental Methods in Algebra, Geometry, and Number Theory.''  Part of the
research for this paper was carried out while the authors were staying at
the Mathematical Research Institute Oberwolfach supported by the ``Research
in Pairs'' programme. The second author wishes to acknowledge support from
Science Foundation Ireland.


\bigskip

\bibliographystyle{plain}

\begin{thebibliography}{10}

\bibitem{barcelobergeron:orlik-solomon}
H.~Barcelo and N.~Bergeron.
\newblock The {O}rlik-{S}olomon algebra on the partition lattice and the free
  {L}ie algebra.
\newblock {\em J. Combin. Theory Ser. A}, 55(1):80--92, 1990.

\bibitem{bergeronbergeron:orthogonal}
F.~Bergeron and N.~Bergeron.
\newblock Orthogonal idempotents in the descent algebra of {$B_n$} and
  applications.
\newblock {\em J. Pure Appl. Algebra}, 79(2):109--129, 1992.

\bibitem{bergeronbergerongarsia:idempotents}
F.~Bergeron, N.~Bergeron, and A.M. Garsia.
\newblock Idempotents for the free {L}ie algebra and {$q$}-enumeration.
\newblock In {\em Invariant theory and tableaux ({M}inneapolis, {MN}, 1988)},
  volume~19 of {\em IMA Vol. Math. Appl.}, pages 166--190. Springer, New York,
  1990.

\bibitem{bergeronbergeronhowletttaylor:decomposition}
F.~Bergeron, N.~Bergeron, R.B. Howlett, and D.E. Taylor.
\newblock A decomposition of the descent algebra of a finite {C}oxeter group.
\newblock {\em J. Algebraic Combin.}, 1(1):23--44, 1992.

\bibitem{bergeron:hochschild}
N.~Bergeron.
\newblock Hyperoctahedral operations on {H}ochschild homology.
\newblock {\em Adv. Math.}, 110(2):255--276, 1995.

\bibitem{bishopdouglasspfeifferroehrle:computations}
M.~Bishop, J.M. Douglass, G.~Pfeiffer, and G.~R{\"o}hrle.
\newblock Computations for {C}oxeter arrangements and {S}olomon's descent
  algebra: Groups of rank three and four.
\newblock {\em Journal of Symbolic Comput.}, 
50:139--158, 2013. 

\bibitem{bishopdouglasspfeifferroehrle:computationsII}
M.~Bishop, J.M. Douglass, G.~Pfeiffer, and G.~R{\"o}hrle.
\newblock Computations for {C}oxeter arrangements and {S}olomon's descent
  algebra {II}: Groups of rank five and six.
\newblock available at {\tt http://arxiv.org/abs/1201.4775}.

\bibitem{blairlehrer:cohomology}
J.~Blair and G.I. Lehrer.
\newblock Cohomology actions and centralisers in unitary reflection groups.
\newblock {\em Proc. London Math. Soc. (3)}, 83(3):582--604, 2001.

\bibitem{brandt:free}
A.~Brandt.
\newblock The free {L}ie ring and {L}ie representations of the full linear
  group.
\newblock {\em Trans. Amer. Math. Soc.}, 56:528--536, 1944.

\bibitem{brieskorn:tresses}
E.~Brieskorn.
\newblock Sur les groupes de tresses [d'apr\`es {V}. {I}. {A}rnold].
\newblock In {\em S\'eminaire Bourbaki, 24\`eme ann\'ee (1971/1972), Exp. No.
  401}, pages 21--44. Lecture Notes in Math., Vol. 317. Springer, Berlin, 1973.

\bibitem{douglass:cohomology}
J.M. Douglass.
\newblock On the cohomology of an arrangement of type ${B}_l$.
\newblock {\em J. Algebra}, 146:265--282, 1992.

\bibitem{douglasspfeifferroehrle:inductive}
J.M. Douglass, G.~Pfeiffer, and G.~R{\"o}hrle.
\newblock An inductive approach to {C}oxeter arrangements and {S}olomon's
  descent algebra.
\newblock {\em J. Algebraic Combin.}, 35(2):215--235, 2012.

\bibitem{felderveselov:coxeter}
G.~Felder and A.~P. Veselov.
\newblock Coxeter group actions on the complement of hyperplanes and special
  involutions.
\newblock {\em J. Eur. Math. Soc. (JEMS)}, 7(1):101--116, 2005.

\bibitem{garsiareutenauer:decomposition}
A.M. Garsia and C.~Reutenauer.
\newblock A decomposition of {S}olomon's descent algebra.
\newblock {\em Adv. Math.}, 77(2):189--262, 1989.

\bibitem{chevie}
M.~Geck, G.~Hi{\ss}, F.~L{\"u}beck, G.~Malle, and G.~Pfeiffer.
\newblock {CHEVIE} --- {A} system for computing and processing generic
  character tables.
\newblock {\em Appl. Algebra Engrg. Comm. Comput.}, 7:175--210, 1996.

\bibitem{geckmalle:frobenius}
M~Geck and G.~Malle.
\newblock Frobenius--{S}chur indicators of unipotent characters and the twisted
  involution module.
\newblock {\tt arXiv:1204.3590v3}.

\bibitem{geckpfeiffer:characters}
M.~Geck and G.~Pfeiffer.
\newblock {\em Characters of finite {C}oxeter groups and {I}wahori-{H}ecke
  algebras}, volume~21 of {\em London Mathematical Society Monographs. New
  Series}.
\newblock The Clarendon Press,` Oxford University Press, New York, 2000.

\bibitem{hanlon:action}
P.~Hanlon.
\newblock The action of {$S_n$} on the components of the {H}odge decomposition
  of {H}ochschild homology.
\newblock {\em Michigan Math. J.}, 37(1):105--124, 1990.

\bibitem{howlett:normalizers}
R.~B. Howlett.
\newblock Normalizers of parabolic subgroups of reflection groups.
\newblock {\em J. London Math. Soc. (2)}, 21(1):62--80, 1980.

\bibitem{howlettlehrer:duality}
R.B. Howlett and G.I. Lehrer.
\newblock Duality in the normalizer of a parabolic subgroup of a finite
  {C}oxeter group.
\newblock {\em Bull. London Math. Soc.}, 14(2):133--136, 1982.

\bibitem{klyachko:lie}
A.A. Klyachko.
\newblock Lie elements in the tensor algebra.
\newblock {\em Siberian Math. J.}, 15:914--921, 1974.

\bibitem{konvalinkapfeifferroever:centralizers}
M.~Konvalinka, G.~Pfeiffer, and C.~R{\"o}ver.
\newblock A note on element centralizers in finite {C}oxeter groups.
\newblock {\em J. Group Theory}, 14:727--745, 2011.

\bibitem{kottwitz:involutions}
R.~Kottwitz.
\newblock Involutions in {W}eyl groups.
\newblock {\em Represent. Theory}, 4:1--15 (electronic), 2000.

\bibitem{lehrer:poincare}
G.I. Lehrer.
\newblock On the {P}oincar\'e series associated with {C}oxeter group actions on
  complements of hyperplanes.
\newblock {\em J. London Math. Soc. (2)}, 36(2):275--294, 1987.

\bibitem{lehrersolomon:symmetric}
G.I. Lehrer and L.~Solomon.
\newblock On the action of the symmetric group on the cohomology of the
  complement of its reflecting hyperplanes.
\newblock {\em J. Algebra}, 104(2):410--424, 1986.

\bibitem{lusztigvogan:hecke}
G.~Lusztig and D.~Vogan.
\newblock Hecke algebras and involutions in {W}eyl groups.
\newblock to appear in {\sl Bull. Inst. Math. Acad. Sinica (N.S.)}.

\bibitem{orliksolomon:combinatorics}
P.~Orlik and L.~Solomon.
\newblock Combinatorics and topology of complements of hyperplanes.
\newblock {\em Invent. Math.}, 56(2):167--189, 1980.

\bibitem{orliksolomon:coxeter}
P.~Orlik and L.~Solomon.
\newblock Coxeter arrangements.
\newblock In {\em Singularities}, volume~40 of {\em Proc. Symp. Pure Math.},
  pages 269--292. Amer. Math. Soc., 1983.

\bibitem{orlikterao:arrangements}
P.~Orlik and H.~Terao.
\newblock {\em Arrangements of Hyperplanes}.
\newblock Springer-Verlag, 1992.

\bibitem{reutenauer:free}
C.~Reutenauer.
\newblock {\em Free {L}ie algebras}, volume~7 of {\em London Mathematical
  Society Monographs, New Series}.
\newblock The Clarendon Press, Oxford University Press, New York, 1993.

\bibitem{gap3}
M.~Sch{\accent127 o}nert et~al.
\newblock {\em {GAP} -- {Groups}, {Algorithms}, and {Programming} -- version 3
  release 4}.
\newblock Lehrstuhl D f{\accent127 u}r Mathematik, Rheinisch Westf{\accent127
  a}lische Technische Hoch\-schule, Aachen, Germany, 1997.

\bibitem{schocker:hoheren}
M.~Schocker.
\newblock \"{U}ber die h\"oheren {L}ie-{D}arstellungen der symmetrischen
  {G}ruppen.
\newblock {\em Bayreuth. Math. Schr.}, (63):103--263, 2001.

\bibitem{solomon:decomposition}
L.~Solomon.
\newblock A decomposition of the group algebra of a finite {C}oxeter group.
\newblock {\em J. Algebra}, 9:220--239, 1968.

\bibitem{solomon:mackey}
L.~Solomon.
\newblock A {M}ackey formula in the group ring of a {C}oxeter group.
\newblock {\em J. Algebra}, 41(2):255--264, 1976.

\bibitem{stanley:aspects}
R.P. Stanley.
\newblock Some aspects of groups acting on finite posets.
\newblock {\em J. Combin. Theory Ser. A}, 32(2):132--161, 1982.

\bibitem{steinberg:differential}
R.~Steinberg.
\newblock Differential equations invariant under finite reflection groups.
\newblock {\em Trans. Amer. Math. Soc.}, 112:392--400, 1964.

\bibitem{thevenaz:G-algebras}
J.~Th{\'e}venaz.
\newblock {\em {$G$}-algebras and modular representation theory}.
\newblock Oxford Mathematical Monographs. The Clarendon Press Oxford University
  Press, New York, 1995.
\newblock Oxford Science Publications.

\bibitem{wever:ueber}
F.~Wever.
\newblock \"{U}ber {I}nvarianten in {L}ie'schen {R}ingen.
\newblock {\em Math. Ann.}, 120:563--580, 1949.

\end{thebibliography}


\end{document}